%% file: arxiv.tex
\pdfoutput=1
\documentclass[11pt, letterpaper]{article}

\input{everything.tex}

\newcommand{\algD}{1} 
\newcommand{\algR}{2} 
\newcommand{\algE}{3} 
\newcommand{\algS}{4} 
\newcommand{\ones}{{\bf 1}}

\newcommand{\oo}[1]{\quad #1 \quad}

\title{Fastest Rates for Stochastic Mirror Descent Methods\thanks{All theoretical results of this paper were obtained by June 2017. }}
\author{Filip Hanzely\footnote{King Abdullah University of Science and Technology (KAUST), Thuwal, Saudi Arabia} \and Peter Richt\'{a}rik\footnote{King Abdullah University of Science and Technology (KAUST), Thuwal, Saudi Arabia --- University of Edinburgh, Edinburgh, United Kingdom --- Moscow Institute of Physics and Technology (MIPT), Dolgoprudny, Russia}}
\date{March 20, 2018}

\theoremstyle{definition}

\begin{document}
\maketitle

\begin{abstract} 
 Relative smoothness - a notion introduced in \cite{birnbaum2011distributed} and recently rediscovered in \cite{bauschke2016NoLips, Rel_smoothness_nesterov} - generalizes the standard notion of smoothness typically used in the analysis of gradient type methods. In this work we are taking ideas from well studied field of stochastic convex optimization and using them in order to obtain faster algorithms for minimizing relatively smooth functions. We propose and analyze two new algorithms: Relative Randomized Coordinate Descent (relRCD) and Relative Stochastic Gradient Descent (relSGD), both generalizing famous algorithms in the standard smooth setting. The methods we propose can be in fact seen as a particular instances of stochastic mirror descent algorithms. One of them, relRCD corresponds to the first stochastic variant of mirror descent algorithm with linear convergence rate. 
\end{abstract}

\section{Introduction}

During  the last decade or so, {\em first order methods} have become the main algorithmic toolbox for practitioners  solving  optimization problems of large sizes, especially in application domains where low to medium accuracy is  sufficient. These methods are now the state of the art for many problems arising in areas such as machine learning,  statistics, signal processing, computer vision, inverse problems and data science. Arguably, algorithms for {\em smooth convex optimization} form the backbone of this new development, and the basis for subsequent extensions beyond convexity and smoothness.

In this paper we consider the optimization problem
\begin{eqnarray}
\min  && f(x)
\label{Eq:optimization_problem}
\\
\nonumber
\text{subject to} && x\in Q,
\end{eqnarray}
where $Q\subseteq \R^n$ is a closed convex set, and $f$ is  a convex and differentiable\footnote{We assume that $f$ is differentiable on some open set containing $Q$.} (objective/loss) function.

Our work is motivated by the need to solve problems of the form \eqref{Eq:optimization_problem} in the ``big data'' regime, that is, in situations when either the dimensionality of the problem, $n$, is very large, or when $f$ is of a finite sum structure, \begin{equation} \label{eq:kjhbd78ge8y}f(x)=\frac{1}{m}\sum_{i=1}^m f_i(x),\end{equation}
with the number of components, $m$, being very large. In particular, we are interested in designing efficient {\em randomized} first order methods for  \eqref{Eq:optimization_problem} without the need to assume for $f$ to have Lipschitz gradients, thus extending the reach of modern randomized gradient-type methods to new territories.

\subsection{Lipschitz continuity of the gradients}

It is remarkable that virtually the entire development of first order methods for smooth convex optimization hinges on what turns out to be a very restrictive regularity assumption on the behaviour of the gradients of $f$, thus preventing their applicability to domains where this assumption does not hold, or is unreasonable due to practical considerations. In particular, it is universally assumed for the objective function $f$ to have Lipschitz continuous  gradients \cite{nemirovsky1983problem, poljak1987introduction, NesterovBook}.  Recall that $f$ is said to be {\em $L$-smooth} on $Q$ (equivalently, we say that  {\em the gradient of $f$ is $L$-Lipschitz} on $Q$), if \begin{equation}\label{eq:iuhosi8hif}
f(x) \leq f(y) + \langle \nabla f(y), x-y\rangle + \frac{L}{2}\|x-y\|^2, \quad \text{for all} \quad x,y\in Q,\end{equation}
where $\langle \cdot,\cdot \rangle$ is an inner product and $\|x\|=\langle x,x\rangle ^{1/2}$ is the induced norm\footnote{An  equivalent characterization of $L$-smoothness is to require the inequality $\|\nabla f(x)-\nabla f(y)\| \leq L\|x-y\|$ to hold for all $x,y\in Q$. }. 

The archetypal first order method for solving \eqref{Eq:optimization_problem}, {\em projected gradient descent} (GD), is  designed to take  advantage of the approximation \eqref{eq:iuhosi8hif}. Given  $x_t\in Q$, the next iterate $x_{t+1}$ of GD is obtained by minimizing the upper {\em quadratic} bound on $f$ provided by \eqref{eq:iuhosi8hif} for $y=x_t$:
\[\boxed{x_{t+1} = \arg\min_{x\in Q} \; \langle \nabla f(x_t), x-x_t \rangle + \frac{L}{2}\|x-x_t\|^2}\]
That is, in the design of GD, one employs a majorize-minimize approach \cite{Lange-MM-book}.

However, there are many differentiable convex functions  which are not $L$-smooth for any finite $L$. For instance, consider the function $f(x)=x^4$ on $\R$. If we still wish to apply a gradient type method to minimize such a function, $L$-smoothness can sometimes be forced upon $f$ by introducing appropriate constrains. This is sufficient in principle as the theory for constrained first order methods only requires the gradients to be $L$-Lipschitz on the domain of interest. However,  such a restriction often leads to a very large constant $L$ in practice, which leads to a prohibitive slow-down of the methods, unless line search strategies are used. Indeed,  the performance of first order type methods deteriorates as $L$ grows, typically  at a linear or quadratic rate. Moreover, even if the objective is naturally $L$-smooth, the constant $L$ is often very large, reflecting poor conditioning of the  problem. In all these cases, direct application of first-order machinery is either impossible or prohibitively inefficient, which leaves these problems beyond the reach of some of the most efficient algorithms designed for large problems in the last decade.

\subsection{Relative smoothness: beyond Lipschitz continuity} 

Relative smoothness was first introduced in~\cite{birnbaum2011distributed} and later rediscovered  independently~\cite{bauschke2016NoLips} and~\cite{Rel_smoothness_nesterov}

This notion enables to design and analyze a generalized version of GD which we  refer to in this paper as {\em relative gradient descent} (relGD). We shall now briefly outline their approach.

Let $h:Q\to \R$ be a strictly convex and differentiable function. The Bregman distance (divergence) of $h$ is the function
\begin{equation}
D_h(x,y) \eqdef h(x)-h(y)-\langle \nabla h(y),x-y\rangle. 
\label{eq:bregman_def}
\end{equation}
Clearly, $D_h(x,y)\geq 0$ and $D_h(x,y)=0$ if and only if $x=y$. However, $D_h$ is not necessarily symmetric. 

 In analogy with \eqref{eq:iuhosi8hif}, Bauschke et al~\cite{bauschke2016NoLips} and Lu et al~\cite{Rel_smoothness_nesterov} say that $f$ is {\em $L$-smooth relative to $h$ on $Q$}  if
\begin{equation}\label{eq:relsmooth}
f(x) \leq f(y) +\langle \nabla f(y),x-y \rangle+LD_h(x,y), \quad \text{for all}\quad x,y\in Q.
\end{equation}


In analogy with the design of gradient descent, relative gradient descent minimizes the upper bound on $f$ given by \eqref{eq:relsmooth} for $y=x_t$:
\begin{equation}\label{eq:iug89fg98he} \boxed{x_{t+1} = \arg\min_{x\in Q} \; \langle \nabla f(x_t), x-x_t \rangle + L D(x,x_t)}\end{equation}

Note that if $h(x)=\frac12\|x \|^2$, then     $D_h(x,y)=\frac{1}{2}
\|x-y \|^2$, and $L$-smoothness relative to $h$ defined in \eqref{eq:relsmooth} coincides with standard $L$-smoothness defined in  \eqref{eq:iuhosi8hif}. Likewise, relative gradient descent coincides with gradient descent.



\subsection{Introducing randomness}

For problems of truly huge sizes (if, as alluded to earlier, either $m$ or $n$ are very large),  {\em randomized} first order methods, such as variants of stochastic gradient descent \cite{RobbinsMonro:1951,nemirovski2009robust, shai_book, sag, SDCA,  kingma2014adam} (in case of large $m$) and randomized coordinate descent (in case of large $n$) \cite{nesterov2012efficiency, UCDC, PCDM, SDCA, ALPHA}, have become the methods of choice, both in theory and in practice.
 
While a single  iteration of a randomized method typically leads to small improvement relative to the improvement obtained by its deterministic counterpart, stochastic iterations are in general much faster: for problems of suitable structure, each iteration is typically $n$ (for randomized coordinate descent type methods) or $m$ (for stochastic gradient descent type methods) times faster than one iteration of gradient descent. The trade-off is in favour of stochastic methods: the savings obtained by performing faster iterations outweigh the loss incurred by settling with smaller per-iteration improvements.


\subsection{Contributions}

In this paper we develop the first stochastic algorithms for minimizing relatively smooth functions. In so doing, we push the boundary of  big data optimization beyond the realm of $L$-smoothness.

All methods developed in this work are of the form
\begin{equation}
\boxed{x_{t+1} = \argmin_{x\in Q_t}\left\{\langle g_t,x\rangle+L_tD_h(x,x_t)\right\}}
\label{next^{(i)}terate}
\end{equation} 
for suitable set $Q_t\subset \R^n$, vector $g_t \in \R^n$ and a sequence of stepsizes $\{L_t\}$. Note that by choosing $g_t=\nabla f(x_t)$, $L_t=L$ and $Q_t=Q$, we obtain method \eqref{eq:iug89fg98he}, i.e., relative gradient descent \cite{bauschke2016NoLips,Rel_smoothness_nesterov}.

We prove convergence of different success measures, including expected suboptimality in the objective, Bregman distance to the optimum, and Bregman distance between iterates. Below we briefly outline some of the results obtained.

Our algorithms  belong to two categories:

\paragraph{Relative Randomized Coordinate Descent (relRCD).} This arises as a special  of the generic method \eqref{next^{(i)}terate} if we choose $g_t=\nabla f(x_t)$, pick suitable stepsizes $L_t$, and let $Q_t$ correspond to a search space generated by  a  random subset of coordinates chose at iteration $t$.  This work can be seen as combining some of the ideas contained in  works on parallel/minibatch coordinate descent \cite{PCDM, NSync, ALPHA} and extending them to the relatively smooth setting. 

We first introduce a basic variant, which uses conservative (small) stepsizes $L_t=L$ (for $Q=\R^n$ thsi would result in stepsize $1/L$). We then perform a more detailed analysis by introducing an  ESO (expected  separable overapproximation) inequality \cite{PCDM, NSync, ESO} applicable to relatively smooth functions. This allows us to choose larger stepsizes $L_t\leq L$, leading to better convergence rates. 
In particular, under a relative strong convexity assumption (see Equations \eqref{eq:relsc} and \eqref{eq:sc_vec_def} for the definition), we obtain the rate (see Theorem~\ref{theorem_rcd_eso}) \[\left(1-p_0\min_{i=1,2,\dots,n} \tfrac{v^{(i)}}{w^{(i)}}\right)^t,\] where $p_0 = \tau/n$ is the probability that we sample any particular coordinate at each iteration, $\tau$ is the number of coordinates sampled in each iteration, $v^{(i)}$ are ESO parameters (we always have $v^{(i)} \leq L$), and $w^{(i)}$ are relative strong convexity parameters.  This rate is the same as the one in~\cite{NSync} which applies to standard randomized coordinate descent, i.e.,  without relative smoothness. On the other hand, if we choose $\tau=n$, we recover relative gradient descent, and the above rate recovers the rate obtained in~\cite{bauschke2016NoLips, Rel_smoothness_nesterov}.

As we show through numerical experiments, relRCD can be much faster than relGD. 


\paragraph{Relative Stochastic Gradient Descent (relSGD).} 
This is a special case of the generic method \eqref{next^{(i)}terate} if we choose $g_t$ to be an unbiased estimator of $\nabla f(x_t)$, $L_t\geq L$, and $Q_t=Q$.  This method extends the applicability of  stochastic gradient descent to the  relatively smooth setting. Convergence of the algorithm is obtained by using a specific decreasing stepsize rule (see Lemmas~\ref{L: sgd good choice of stepsizes} and \ref{l:sgd_dec_no_mu}). With suitable choice of stepsizes, we obtain $O(1/t)$ convergence rate under  relative strong convexity,  and $O(1/\sqrt{t})$ under relative smoothness alone. The rates we obtain generalized the rates known for standard stochastic gradient descent \cite{shai_book}.


%
%
%
%
%
%

\subsection{Related work on relative smooth optimization} Relative smoothness was first introduced in \cite{birnbaum2011distributed} and later rediscovered in \cite{bauschke2016NoLips} and \cite{Rel_smoothness_nesterov} following other works \cite{benning2016Relative_smoothness, bach_rel_smooth2017}. In \cite{birnbaum2011distributed}, Fisher market equilibrium problem was tacked and it was shown that a known algorithm to solve it, proportional response dynamics, is a special instance of relative gradient descent under relative smoothness assumption \cite{zhang_prd}.
In \cite{bauschke2016NoLips} the focus is on minimizing a composite objective, $f(x)+g(x)$, where $f$ is relatively smooth and convex, and
$g$ is convex but not necessarily differentiable. The first proximal algorithm in the relatively smooth setting is proposed there. In \cite{Rel_smoothness_nesterov}, the authors  introduce the notion of relative strong convexity, and propose a dual averaging scheme. In \cite{benning2016Relative_smoothness}, the  authors show that their algorithm converges to a stationary point for nonconvex $f$; no rates are given. Finally, in~\cite{bach_rel_smooth2017}, the authors extend the ideas of dual averaging to stochastic dual averaging. However, this is  only done for quadratic $f$. 

We should also menttion that there is a recent extension of minimizing relative continuous functions \cite{rel_continuity} where Lipschitzness assumption was generalized analogously as smoothness is extended by relative smoothness, opening up a new area of algorithms and applications.

\subsection{Mirror Descent}

Notice that the update rule \eqref{eq:iug89fg98he} of relative gradient descent coincides with mirror descent update rule \cite{nemirovsky1983problem, beck2003mirror}. Therefore, from practical perspective, relative gradient descent enjoys all advantages of mirror descent. 

Let us now briefly review a recent mirror descent literature. We identify two main streams of work on mirror descent. 

One focuses on accelerating deterministic Mirror descent using Nesterov's idea \cite{nesterov1983method}. A significant contribution in this are was done in \cite{tseng2008accelerated}, where previous methods were unified, and couple of enw ones were discovered. A novel approach using the insights from ODE's can be found in \cite{krichene2015accelerated}. In both cases, sublinear $\cO(k^{-2})$ rates were obtained and $f$ was assumed to be smooth convex respectively. There is also a recent work on acceleration using coupling mirror and gradient descent \cite{allen2014linear}, resulting in $\cO(k^{-2})$ rate as well. However, to best of our knowledge, no linear rates for mirror descent are known, except of ones in the relative smooth setting. 
 
The second stream sucuses on stochastic mirror descent with access to noised gradient oracle. In~\cite{nemirovski2009robust, ghadimi2012optimal,nedic2014stochastic} stochastic subgradient mirror descent was considered with $\cO(k^{-1})$ rate for strongly convex and $\cO(k^{\frac12})$ for nonstrongly convex functions. The convergence was obtained using decreasing stepsize in this case and considering bounded variance. 
An accelerated stochastic mirror descent dedicated for ERM problems was proposed in \cite{hien2016accelerated}, obtaining $O(k^{-2})$ convergence rate for smooth convex but non strongly convex functions. 
There is also a very limited literature on coordinate mirror descent strategies. In \cite{afkanpour2013randomized}, coordinate mirror descent was designed for multiple kernel learning problems. The method was casted as a special instance of stochastic mirror descent, obtaining $\cO(k^{-1})$ convergence rate.  Later in
\cite{dang2015stochastic},  stochastic block mirror descent -- where the randomness appears from both coordinate choice and noised gradient was considered, obtaining $\cO(1/k)$ rate for strongly convex and $\cO(k^{\frac12})$ for nonstrongly convex functions. Again, variance of the stochastic gradients was assumed to be bounded here. 

To compare with our results, we stress that relative smoothness setting allows mirror map to be non-strongly convex, in contrast to virtually whole mirror descent literature. On top of that, it allows to obtain linear rates due to the (relative) strong convexity. In general, relative smooth setting allows mirror descent to be directly compared to standard gradient descent. 
In particular, to best of our knowledge, we develop the first stochastic mirror descent algorithm -- relRCD -- with linear convergence rate which outperforms relGD. The setup for relRCD is similar to randomized coordinate descent setup \cite{NSync}, but different to the coordinate mirror descent strategies mentioned above, as we do not consider or enforce stochastic gradient estimates, rather we take gradient descent step in batch of coordinates with stepsize determined from smoothness\footnote{In fact, stepsize is determined from ESO assumption as in \cite{NSync}, which we explain in  Section~\ref{S:eso}}. 
Our other contribution -- relSGD -- is also an extension of stochastic gradient descent in standard smooth setting. We obtain very similar rates comparing to standard mirror descent literature, however the setting we consider is different -- we consider (relatively) smooth problems in contrast to \cite{nemirovski2009robust, ghadimi2012optimal,nedic2014stochastic}, where nonsmooth problems are tackled.

\section{Relatively Smooth Functions and Relative Gradient Descent}
\label{S:rss}

In this section, we introduce relative strongly convex property and give equivalent conditions on both relative smoothness and relative strong convexity. We also mention here a minimization algorithm under the relative smooth assumption - Relative Gradient Descent.

\subsection{Relative smoothness and relative strong convexity}

We firstly start by defining relative strong convexity, which is together with relative smoothness a key assumption for determining a convergence rate of algorithms mentioned in this work. Recall that we defined relative smoothness previously in \eqref{eq:relsmooth}.

\begin{definition} (Relative strong convexity \cite{Rel_smoothness_nesterov})
Function $f$ is $\mu$--strongly convex relative to $h$ on $Q$ if for any $x,\,y\in Q$ the following inequality holds  

\begin{equation}\label{eq:relsc}
f(y) \oo{\geq} f(x) +\langle \nabla f(x),y-x \rangle+\mu D_h(y,x).
\end{equation}
\end{definition}

As the main goal of this work is to minimize function $f$, we have freedom of choice of reference function $h$ - and one would like to choose it so that the convergence rate we obtain is the best possible. In particular, as mentioned in the introduction, for a specific choice $h(x)=\frac12\|x \|^2$ we have $D_h(x,y)=\frac12\|x-y \|^2$ and relative strong convexity assumption becomes standard strong convexity.


The following results  list some elementary properties of relative smooth functions. 

\begin{proposition}[\cite{bauschke2016NoLips,Rel_smoothness_nesterov}]
The following statements are equivalent:
\begin{itemize}
\item $f$ is $L$--smooth relative to $h$ on $Q$ 
\item $Lh(x)-f(x)$ is a convex function on $Q$
\item Under twice differentiability $L\nabla^2h(x) \succcurlyeq \nabla^2f(x) $ for all $x\in Q$
\item $\langle \nabla f(x)-\nabla f(y),x-y \rangle \leq L\langle \nabla h(x)-\nabla h(y),x-y \rangle$ for all $x\in Q$
\end{itemize}
\end{proposition}

For completeness, we also list of equivalent conditions to relative strong convexity.

\begin{proposition}[\cite{bauschke2016NoLips,Rel_smoothness_nesterov}]
The following statements are equivalent:
\begin{itemize}
\item $f$ is $\mu$--strongly convex relative  to $h$ on $Q$ 
\item $f(x)-\mu h(x)$ is a convex function on $Q$
\item Under twice differentiability $\nabla^2f(x) \succcurlyeq  \mu\nabla^2h(x)$ for all $x\in Q$
\item $\langle \nabla f(x)-\nabla f(y),x-y \rangle \geq \mu\langle \nabla h(x)-\nabla h(y),x-y \rangle$ for all $x\in Q$
\end{itemize}
\end{proposition}

The second (convexity) and third (Hessians) conditions appearing in the two propositions above are typically easier  to verify in practice.  For proofs of the propositions, more properties of relatively smooth functions and some examples, we refer the reader to \cite{bauschke2016NoLips,Rel_smoothness_nesterov}. 

\subsection{Relative gradient descent}

Now we are ready to write relative gradient descent (relGD) - baseline algorithm for minimizing relatively smooth functions.

\begin{algorithm}[H]
\textbf{Input: }{Initial iterate $x_0$; reference function $h$ and constant $L>0$ such that $f$ is $L$--smooth relative to $h$.}\\
\For {$t= 0,1,\dots, k-1$} {
	 \begin{enumerate}
\item Set $x_{t+1}\leftarrow \argmin_{x\in Q}\left\{\langle \nabla f(x_t),x\rangle+LD_h(x,x_t)\right\}$ 
\end{enumerate}
 }
\textbf{return} $x_k$
\caption{relGD (Relative Gradient Descent) \cite{birnbaum2011distributed, bauschke2016NoLips, Rel_smoothness_nesterov}}
\end{algorithm}

As mentioned in the introduction, if $h(x)=\tfrac12\|x\|^2$ and $Q=\R^n$, we have \[
x_{t+1} \oo{=} \argmin_{x\in Q} \left\{\langle\nabla f(x_t),x\rangle+\frac{L}{2}\|x-x_t\|^2 \right\} \oo{=} x_t-\frac1L \nabla f(x_t),
\]
and relGD coincides with standard gradient descent with stepsize $\tfrac1L$.

Note also that Algorithm~\algD\ is identical to Mirror descent \cite{beck2003mirror}. The difference that we do not assume standard smoothness but relative smoothness with reference function $h$, thus the analysis and convergence results are significantly different.


The analysis of the algorithm is similar to the analysis of gradient descent under the smoothness assumption. The main difference is that one can explicitly compute the decrease in objective which is guaranteed from the standard smoothness property. This is not the case for the relative smooth optimization as we do not have a general closed expression for the next iterate. In order to overcome this issue, we are using so called three point property \cite{Three_point_porperty_Lan}. This is not a novel approach, it was used in \cite{bauschke2016NoLips, Rel_smoothness_nesterov}. As we need to bound the guaranteed decrease in objective, the analysis becomes slightly looser, which is a price for the generality. However, as we show later, one can still obtain the same convergence result on the ``$O$'' notation comparing to the standard smooth setting.

\begin{lemma}[Three point property]
\label{L:tpp}
Let $\phi, h$ be differentiable convex functions both defined on some convex set $Q$. Let $D_h(\cdot,\cdot)$ be a Bregman distance. For a given $z \in Q$ denote
$$z_+ \oo{\eqdef} \argmin_{x\in Q} \phi(x)+D_{h}(x,z).$$
Then
\begin{equation}
\phi(x)+D_{h}(x,z) \oo{\geq} \phi(z_+)+D_{h}(z_+,z)+D_{h}(x,z_+),\quad \forall x\in Q.
\label{eq:tpp}
\end{equation}
\end{lemma}

Proof of the three point property can be found in the appendix. The following theorem states a convergence result of relative gradient descent. 

\begin{theorem}[Lu, Freund and Nesterov~\cite{Rel_smoothness_nesterov}]\label{t:primal_scheme}
Consider  Algorithm~\algD. If $f$ is $L$--smooth and $\mu$--strongly convex
relative to $h$ for some $L > 0$ and $\mu \geq 0$, then for all $k\geq 1$ the following inequality
holds:
\[
f(x_k)-f(x_*)
\oo{\leq}
 \frac{\mu D_h(x_*,x_0)}{\left(1+\frac{\mu}{L-\mu}\right)^k-1}
 \oo{\leq }
\frac{L-\mu}{k}D_h(x,x_0).
\]
 In the case when $\mu = 0$, the middle expression is defined in the limit as $\mu \rightarrow 0_+$.
\end{theorem}

In the case when $\mu>0$, Relative Gradient Descent enjoys linear convergence rate, which is asymptotically driven by
\[
\left(1+\frac{\mu}{L-\mu}\right)^{-k}
=
\left(\frac{L}{L-\mu}\right)^{-k}
=\left(1-\frac{\mu}{L}\right)^k
.\] On the other hand if $\mu=0$, Theorem~\ref{t:primal_scheme} yields $O(1/k)$ convergence rate. Thus, relative gradient descent is, up to the constant term, matching rate of standard Gradient descent under standard smoothness assumption.

\section{Relative Randomized Coordinate Descent with Short Stepsizes \label{S:rcd}}

In this section, we propose and analyze  a naive coordinate descent algorithm for minimizing relative smooth functions. The key idea is to choose a subset of coordinates each iteration and make a step from relGD in the corresponding subspace. 

We give two slightly different ways to analyze the convergence. However, neither of them provides a speedup comparing to Algorithm~\algD\,. We mention this for educational purposes, to illustrate our techniques. This issue will be adressed later in Section~\ref{S:eso}, providing us a potential speedup comparing to Algorithm~\algD. 

The key assumption of this section - separability is defined in the following way:
$ h(x)=\sum_{i=1}^n h^{(i)}\left(x^{(i)}\right),$
where $h^{(i)}$ takes only $i$-th coordinate of $x$. On top of that, we assume that $Q$ is block separable: $Q=\prod_{i=1}^nQ^{(i)}$ where $Q^{(i)}$ is closed interval for all $i$. In other words $x\in Q$ if and only if for all $i$ we have $x^{(i)}\in Q^{(i)}$.

Throughout this section, we assume that $f$ is $L$--smooth and $\mu$--strongly convex relative to some separable function $h$.

\subsection{Algorithm}

We introduce here Algorithm~\algR\, -- Relative Randomized Coordinate descent with short stepsizes. From now, let us denote $\ones^i$ to be $i$--th column of $n\times n$ identity matrix. The update is given by \eqref{next^{(i)}terate} with 
\[Q_t=\left\{x \;\Big|\; x=x_t+\sum_{i\in M_t}\mathrm{span}\left(\ones^i\right) \right\}. \] 

Subset of coordinates $M_t$ is chosen randomly such that $\Prob(i\in M_t)=\Prob(j\in M_t)$ for all $i,j\leq n$ and $|M_t|=\mb$. 

\begin{algorithm}[H]
\textbf{Input: }{Initial iterate $x_0$, separable reference function $h$ and $L$ such that $f$ is $L$--smooth relative to $h$.}\\
\For {$t= 0,1,\dots, k-1$} {
	 \begin{enumerate}
\item Choose $M_t\subseteq \{1,2,\dots, n\}$ such that $\Prob(i\in M_t)=\Prob(j\in M_t)$ for all $i,j\leq n$ and $|M_t|=\mb$
\item Set $Q_t\leftarrow \left\{x\;|\;x=x_t+\sum_{i\in M_t}\mathrm{span}\left(\ones^i\right) \right\}$ 
\item Set $x_{t+1}\leftarrow \argmin_{x\in Q_t} \left\{\langle\nabla f(x_t),x\rangle+LD_h(x,x_t) \right\}$ 
\end{enumerate}
 }
\textbf{return} $x_k$

\caption{ relRCDs (Relative Randomized Coordinate Descent with Short Stepsizes)}

\end{algorithm}

\subsection{Key lemma}

It will be useful to introduce \[x_{(t+1,*)}\oo{\eqdef}\argmin_{x\in Q}\{\langle \nabla f(x_t),x \rangle +LD_{h}(x,x_t)\}\] as we will use this notation in the analysis.

The following lemma describes behavior of Algorithm~\algR\ in each iteration, providing us on the expected upper bound on the value in the next iterate using the previous iterate. 

\begin{lemma}[Iteration decrease for Algorithm~\algR]
Suppose that $f$ is $L$--smooth and $\mu$--strongly convex relative to separable function $h$. Then, running Algorithm~\algR\ we obtain for all $x\in Q$:
\begin{equation*}
\Exp{f(x_{t+1})}\oo{\leq} \frac{n-\mb}{n}\Exp{ f(x_{t})}+ \frac{\mb}{n}  f(x)+\left(L-\frac{\mb}{n}\mu\right)\Exp{D_{h}(x,x_t)}-L\Exp{D_{h}(x,x_{t+1})}.
\end{equation*}
\label{iter_decrease_mb}

\end{lemma}

\begin{proof}

\begin{eqnarray}
\Exp{f(x_{t+1})|x_t} 
& \stackrel{\eqref{eq:relsmooth}}{\leq} & 
f(x_{t})+ \Exp{\Big ( \langle \nabla f(x_t), x_{t+1}-x_t \rangle +LD_{h}(x_{t+1},x_t)\Big )\;|\;x_t}
\nonumber
\\
&=&
\nonumber
f(x_{t})+ \Exp{\sum_{i\not\in M_t}\Big (\left(\nabla f(x_t)\right)^{(i)}( x_{t+1}-x_t)^{(i)} +LD_{h^{(i)}}\left(x_{t+1}^{(i)},x_t^{(i)}\right)\Big )\;|\;x_t} 
\\
&& 
\qquad +
 \Exp{\sum_{i\in M_t}\Big (\left(\nabla f(x_t)\right)^{(i)}(x_{t+1}-x_t)^{(i)} +LD_{h^{(i)}}\left(x_{t+1}^{(i)},x_t^{(i)}\right)\Big )\;|\;x_t}
 \nonumber
\\
& \stackrel{(*)}{=} & 
\nonumber
f(x_{t}) + \Exp{\sum_{i\in M_t}\Big (\left(\nabla f(x_t)\right)^{(i)}(x_{t+1}-x_t)^{(i)} +LD_{h^{(i)}}\left(x_{t+1}^{(i)},x_t^{(i)}\right)\Big )\;|\;x_t}
\\
& = & 
\nonumber
f(x_{t}) +  \frac{\mb}{n}\left \langle\nabla f(x_t), x_{(t+1,*)}-x_t\right \rangle +\frac{\mb}{n}LD_{h}\left(x_{(t+1,*)},x_t\right)
\\
& \stackrel{\eqref{eq:tpp}}{\leq} & 
\nonumber
f(x_{t}) +  \frac{\mb}{n} \langle \nabla f(x_t) ,x-x_t \rangle +\frac{\mb}{n} LD_{h}(x,x_t) -\frac{\mb}{n} L D_h(x,x_{(t+1,*)})
\\
& \stackrel{\eqref{eq:relsc}}{\leq} &  
\frac{n-\mb}{n} f(x_{t})+ \frac{\mb}{n} f(x) +   \frac{\mb}{n} (L-\mu)D_{h}(x,x_t)-\frac{\mb}{n} L D_h(x,x_{(t+1,*)}).
\label{eq:rcd^{(i)}ter_mb_1}
\end{eqnarray}

The equality $(*)$ above holds due to the fact that $x_{t+1}^{(i)}=x_{t}^{(i)}$ for $i\not \in M_t$. 
Note that  
\[
\Exp{D_{h}(x,x_{t+1})\;|\;x_t}
\oo{=}
\frac{n-\mb}{n}D_h(x,x_t)+  \frac{\mb}{n} D_h(x,x_{(t+1,*)})
.\] 

Plugging it into \eqref{eq:rcd^{(i)}ter_mb_1}, we get

\begin{eqnarray*}
\Exp{f(x_{t+1})|x_t}
&\stackrel{ \eqref{eq:rcd^{(i)}ter_mb_1}}{\leq}& 
\frac{n-\mb}{n} f(x_{t})+ \frac{\mb}{n}  f(x)+\frac{\mb}{n}(L-\mu)D_{h}(x,x_t)-L\Exp{D_{h}(x,x_{t+1})|x_t}
\\
&& \qquad
+\frac{n-\mb}{n}LD_h(x,x_t))
\\
&=&
\frac{n-\mb}{n} f(x_{t})+ \frac{\mb}{n}  f(x)+\left(L-\frac{\mb}{n}\mu\right)D_{h}(x,x_t)-L\Exp{D_{h}(x,x_{t+1})|x_t}.
\end{eqnarray*}

Taking the expectation over the algorithm and using the tower property we obtain the desired result. 
\end{proof}

The lemma above provides us with the expected decrease in the objective every iteration. It holds for all $x\in Q$, particularly for $x=x_t$ we obtain that the sequence $\{f(x_t)\}$ is nonincreasing in expectation.

\subsection{Strongly convex case: $\mu>0$}

The following theorem uses recursively Lemma~\ref{iter_decrease_mb} with $x=x_*$, obtaining a convergence rate of Algorithm~\algR\,.

\begin{theorem}[Convergence rate for Algorithm~\algR ]
Suppose that $f$ is $L$--smooth and $\mu$--strongly convex relative to separable function $h$ for $\mu>0$. Running Algorithm~\algR\ for $k$ iterations we obtain:
\begin{equation*}
\sum_{t=1}^{k}c_t \big(\Exp{f(x_t)}-f(x_*)\big)
\leq
    \frac{(L-\frac{\mb}{n}\mu)D_{h}(x_*,x_0) +\frac{n-\mb}{n}(f(x_0)-f(x_*))}{1-\frac{L}{\mu}+\frac{L}{\mu}\Big ( \frac{L}{L-\frac{\mb}{n}\mu} \Big )^{k-1}}, 
\end{equation*}
where $c= (c_1,\dots,c_k)\in \R^k_+$ is a positive vector with entries summing up to 1.
\label{theorem_rcd_mb}
\end{theorem}

\begin{proof}
The proof follows by applying Lemma \ref{l:rate_from_iter} on Lemma \ref{iter_decrease_mb} with $x=x_*$ for 
$f_t=\Exp{f(x_t)},\,D_t=\Exp{D_h(x_*,x_t)},\,f_*= f(x_*),\, \delta=\tfrac{\mb}{n},\,\ \varphi=L,\, \psi=\mu$.

\end{proof}

Note that the term driving the convergence rate in Theorem~\ref{theorem_rcd_mb} is  
$
\left( L/(L-\tfrac{\mb}{n}\mu)\right)^{1-k}
=
\left(1-\tfrac{\mb}{n}\tfrac{\mu}{L}\right)^{k-1},
$
where $k$ is the number if iterations. In the special case when $\mb=n$, using simple algebra one can verify that Theorem~\ref{theorem_rcd_mb} matches the results from Theorem~\ref{t:primal_scheme}. 

\subsection{Non-strongly convex case: $\mu=0$}

The following theorem provides us with the convergence rate of Algorithm~\algR\, when $f$ is convex but not necessarily relative strongly convex (i.e., $\mu=0$).

\begin{theorem}[Convergence rate for Algorithm~\algR ]
Suppose that $f$ is convex and $L$--smooth relative to separable function $h$. Running Algorithm~\algR\ for $k$ iterations we obtain:
\begin{equation*}
\sum_{t=1}^{k}c_t(\Exp{f(x_t)}-f(x_*)) 
\oo{\leq} 
\frac{LD_{h}(x,x_0)+\frac{n-\mb}{n}\left(f(x_0)-f(x_*)\right)}{1+\frac{\mb(k-1)}{n}},
\end{equation*}
where $c=(c_1,\dots,c_k)\in \R^k$ is a positive vector proportional to 
$
\big(\frac{\mb}{n},\,\frac{\mb}{n},\,\dots,\,\frac{\mb}{n},\,1 \big)
$.
\label{theorem_rcd_nsc_mb}
\end{theorem}

\begin{proof}
For simplicity, denote $r_t=\Exp{f(x_t)}-f(x_*)$.
We can follow the proof of Theorem~\ref{theorem_rcd_mb} using Lemma~\ref{l:rate_from_iter} to get the equation \eqref{RCD_tel}, which can be rewritten for $\mu=0$ as follows:
\begin{equation*}
LD_{h}(x,x_0)
\oo{\geq} 
r_{k}+
\frac{\mb}{n}\sum_{t=1}^{k-1}  r_t -\frac{n-\mb}{n}r_0.
\end{equation*}
The inequality above can be easily rearranged as

\begin{equation*}
\frac{LD_{h}(x,x_0)+\frac{n-\mb}{n}r_0}{1+(k-1)\frac{\mb}{n}}
\oo{\geq}  
\frac{1}{1+(k-1)\frac{\mb}{n}}\left(r_{k}+
\frac{\mb}{n}\sum_{t=1}^{k-1}  r_t \right).
\end{equation*}
\end{proof}

As previously, Theorem~\ref{theorem_rcd_nsc_mb} captures known results of Relative Gradient Descent for $\mb=n$ (Theorem~\ref{t:primal_scheme}).

\subsection{Improvements using a symmetry measure}

For completeness, we provide a different analysis of Algorithm~\algR\, using a different power function which is a combination of $f(x_{t})-f(x_*) $ and $ D_h(x_{*},x_t)$. A similar analysis in the standard smooth setting was done in \cite{PCD_complexity}.

It would be useful to define a symmetry measure of Bregman distance here. 

\begin{definition}[Symmetry measure]
Given a reference function $h$, the symmetry measure of $D_h$ is defined by
\begin{equation}\label{eq:alfa}
\alpha(h) \oo{\eqdef} \inf_{x,y} \left\{\frac{D_h(x,y)}{D_h(y,x)} \;\Big| \; x\neq y \right\}.
\end{equation}
\end{definition}

Note that we clearly have $0\leq\alpha(h)\leq 1$. A symmetry measure $\alpha_h$ was also used in \cite{bauschke2016NoLips}. In our case, considering the symmetric measure for $D_h$ would improve the result from the next theorem. However our results does not rely on it and hold even if there is no symmetry present, i.e. $\alpha(h)=0$.

\begin{theorem}[Convergence rate for Algorithm~\algR ]
Suppose that $f$ is $L$--smooth and $\mu$--strongly convex relative to separable function $h$. Denote $Z_t^{L}\eqdef LD_h(x_{*},x_t)+f(x_{t})-f(x_*)$.  Running Algorithm~\algR\ for $k$ iterations we obtain:
$$ 
\Exp{f(x_{k})-f(x_*)}
\oo{\leq} 
\frac{Z_0^{L}}{1+\frac{\mb}{n}k}$$
when $\mu =0$ and 
$$ 
\Exp{Z^{L}_{k}}
\oo{\leq} 
\left(1- \frac{\mb}{n} \frac{\mu}{L} - \frac{\mb}{n}\Big( 1-\frac{\mu}{L} \Big)\frac{\mu \alpha(h)}{\mu \alpha(h)+L}\right)^k Z_0^{L}
 $$
when $\mu > 0$.
\label{theorem_rcd_mb_good}
\end{theorem}

\begin{proof}
 From Lemma~\ref{iter_decrease_mb} we have
\begin{equation}
\Exp{Z_{t+1}^{L}}\oo{\leq} \Exp{Z_t^{L}} - \frac{\mb}{n}\Exp{Z_{t}^{\mu}}.
\label{NoSumApproach}
\end{equation}
If $\mu=0$, we can easily telescope the above and get the following inequality
\begin{equation*}
\Exp{f(x_{k})-f(x_*)}\oo{\leq} Z_0^{L} - \frac{\mb}{n}k\Exp{f(x_{k})-f(x_*)},
\end{equation*}
which leads to 
\begin{equation*}
\Exp{f(x_{k})-f(x_*)}\oo{\leq} \frac{Z_0^{L}}{1+\frac{\mb}{n}k}. 
\end{equation*}

Let us look at the case when $\mu \neq 0$. Firstly note that from relative strong convexity of $f$ combining with definition of the symmetric measure $\alpha(h)$ we have 
\begin{equation}
f(x_t)-f(x_*)
\oo{\geq}
 \mu D_h(x_t,x_*)
 \oo{\geq}
  \mu \alpha(h) D_h(x_*,x_t) .
\label{rcd_z_symm_usage}
\end{equation}

Therefore, \eqref{NoSumApproach} can be rewritten as

\begin{eqnarray*}
\Exp{Z_{t+1}^{L}}
&\stackrel{\eqref{NoSumApproach}}{\leq} &
\Exp{Z_t^{L}} - \frac{\mb}{n}\Exp{Z_{t}^{\mu}}
\\
&=&
 \Exp{Z_t^{L}} - \frac{\mb}{n}\frac{\mu}{L}\Exp{Z_{t}^{L}}-\frac{\mb}{n}\Big( 1-\frac{\mu}{L} \Big) (f(x_t)-f(x_*))
 \\
&= & 
 \Exp{Z_t^{L}} - 
 \frac{\mb}{n}\frac{\mu}{L}\Exp{Z_{t}^{L}}
 -\frac{\mb}{n}\Big( 1-\frac{\mu}{L} \Big)
 \frac{\mu \alpha(h)}{\mu \alpha(h)+L}(f(x_t)-f(x_*))
 \\
 && \qquad 
  -\frac{\mb}{n}\Big( 1-\frac{\mu}{L} \Big)
 \frac{L}{\mu \alpha(h)+L}(f(x_t)-f(x_*)) 
\\
&\stackrel{\eqref{rcd_z_symm_usage}}{\leq}& 
 \Exp{Z_t^{L}} - 
 \frac{\mb}{n}\frac{\mu}{L}\Exp{Z_{t}^{L}}
 -\frac{\mb}{n}\Big( 1-\frac{\mu}{L} \Big)
 \frac{\mu \alpha(h)}{\mu \alpha(h)+L}(f(x_t)-f(x_*))
 \\
 && \qquad
 - 
 \frac{\mb}{n}\Big( 1-\frac{\mu}{L} \Big)
 \frac{L}{\mu \alpha(h)+L}\mu \alpha(h)D_h(x_*,x_t) 
\\
&=& 
 \Exp{Z_t^{L}} - 
 \frac{\mb}{n}\frac{\mu}{L}\Exp{Z_{t}^{L}}
 -\frac{\mb}{n}\Big( 1-\frac{\mu}{L} \Big)\frac{\mu \alpha(h)}{\mu \alpha(h)+L}\Exp{Z_t^{L}}
 \\
 &=&
 \left(1- \frac{\mb}{n} \frac{\mu}{L} - \frac{\mb}{n}\Big( 1-\frac{\mu}{L} \Big)\frac{\mu \alpha(h)}{\mu \alpha(h)+L}\right) \Exp{Z_t^{L}}.
\end{eqnarray*}

Using recursively the inequality above, we get
\begin{equation*}
\Exp{Z^{L}_{k}}
\oo{\leq}
 \left(1- \frac{\mb}{n} \frac{\mu}{L} - \frac{\mb}{n}\Big( 1-\frac{\mu}{L} \Big)\frac{\mu \alpha(h)}{\mu \alpha(h)+L}\right) ^k Z_0^{L}.
\end{equation*}

\end{proof}

Note that as soon as $\alpha(h)=0$, rate from the theorem above is up to the constant same as rate from Theorem~\ref{theorem_rcd_mb} since 
$ \left( L/ (L-\frac{\mb}{n}\mu)\right) ^{-1}
=
1-\frac{\mb}{n}\frac{\mu}{L}. $
However both theorems are measuring a convergence rate for a different quantity. On the other hand, in the best case if $\alpha(h)=1$ we have 
\[
1-\frac{\mb}{n} \frac{\mu}{L} - \frac{\mb}{n}\Big( 1-\frac{\mu}{L} \Big)\frac{\mu }{\mu+L}
\oo{=}
1-\frac{\mb}{n} \frac{\mu}{L} -\frac{\mb}{n} \frac{\mu}{L} \left(1-\frac{2\mu}{L+\mu} \right)
\oo{\geq} 
1-2\frac{\mb}{n}\frac{\mu}{L},
\]
thus the convergence rate we obtained might be up to 2 times faster comparing to rate from Theorem~\ref{theorem_rcd_mb}. Thus the convergence rate is also up to 2 times faster comparing to Theorem~\ref{t:primal_scheme} for the case $\mb=n$ if $\alpha(h)<0$. On the other hand, Theorem~\ref{theorem_rcd_mb_good} provides us with convergence rate of $\Exp{D_h(x_*,x_k)}$,  as the following inequality trivially holds:
$$
\Exp{D_h(x_*,x_k)}\quad\leq\quad \frac{\Exp{Z^{L}_{k}}}{L}.
$$

Suppose that we have a fixed budget on the total work of the algorithm, i.e. we can make only $k/\mb$ iterations. It is a simple exercise to notice that the bound on the suboptimality for Theorems~\ref{theorem_rcd_mb}, \ref{theorem_rcd_nsc_mb} and \ref{theorem_rcd_mb_good} after $k/\mb$ iterations is not getting better when minibatch size $\mb$ is decreasing. We address next section in order to solve this issue.

\section{Relative Randomized Coordinate Descent with Large Stepsizes \label{S:eso}}

This section addresses the issue of the previous section - allowing a better usage of randomness in order to obtain a faster convergence rate comparing to the deterministic setting. Theorem~\ref{theorem_rcd_eso} later in this section is one of two key results (together with the analysis of Relative Stochastic Gradient Descent) of this work. 

As previously, we assume that $h$ is separable function, i.e., $h(x)=\sum_{i=1}^n h^{(i)}\left(x^{(i)}\right)$ and $Q$ is block separable $Q=\prod_{i=1}^nQ^{(i)}$. For notational simplicity, let us define weighted Bregman distance and weighted inner product: 
\begin{eqnarray*}
D_ h(x,y)_v
&\eqdef&
\sum_{i=1}^n 
v^{(i)}
\left(
h^{(i)}\left(x^{(i)}\right)
-h^{(i)}\left(y^{(i)}\right)
- \nabla h^{(i)}\left(y^{(i)}\right)\cdot(x^{(i)}-y^{(i)}) 
\right),
\\
\langle a,b \rangle_p &\eqdef& \sum_{i=1}^n p_ia_ib_i
,
\end{eqnarray*}
where $v,p\in \R^n$ are some positive vectors.

It would be also useful to introduce the separable version of relative strong convexity, as a generalization of Relative Strong Convexity with respect to a separable function $h$, allowing different strong convexity parameters for each coordiante.

\begin{definition}[Relative strong convexity, separable version]
Suppose that $w\in \mathbb{R}^n_+$. Function $f$ is $w$-strongly convex relative to separable function $h$ on $Q$ if for any $x,\,y\in \mathrm{int}(Q)$ the following inequality holds  
\begin{equation}
f(y)\oo{\geq} f(x) +\langle \nabla f(x),y-x \rangle+ D_h(y,x)_w.
\label{eq:sc_vec_def}
\end{equation}
\end{definition}

Throughout this section, we will assume the separable version of relative strong convexity, as it captures relative $\mu$--strong convexity as a special case for $w=\mu\ones$ and might potentially bring a better convergence result.

\subsection{Expected separable overapproximation}

In our analysis, we use $h$--ESO assumption defined below instead of relative smoothness assumption. In the standard smoothness setting, it was firstly introduced in \cite{pcd_takac}.

\begin{definition}[$h$--ESO]
Let $h$ be some separable function and $p=p(\hat{S})$ be a probability vector of the sampling $\hat{S}$, i.e. $p^{(i)}=\Prob(i\in \hat{S})$. Function $f$ admits Expected Separable Overapproximation with respect to function $h$ ($h$--ESO), parameters $\hat{S}$ (probability sampling) and $v$ (vector) if the following inequality holds for all $x,\,q \in \R^n$:
\begin{equation}
\Exp{f\left(x+\sum_{i\in \hat{S}}q^{(i)}\ones^i\right)}
\oo{\leq}
 f(x)+\langle \nabla f(x),q \rangle_p+D_h(x+q,x)_{p\circ v}.
\label{eq:eso_def}
\end{equation}
For simplicity, we write $(f,\hat{S})\sim \mathrm{ESO}_h(v)$.
\end{definition}
Above, ``$\circ$'' denotes Hadamard product, i.e. element-wise product of two vectors. Note that if $f$ is $L$--smooth relative to the separable function $h$, then we have

\begin{eqnarray*}
\Exp{f\left(x+\sum_{i\in \hat{S}}q^{(i)}\ones^i\right)}
&\leq &
\Exp{f(x)+\left \langle \nabla f(x),\sum_{i\in \hat{S}}q^{(i)}\ones^i \right \rangle + LD_h\left(x+\sum_{i\in \hat{S}}q^{(i)}\ones^i,x\right)}
\\
&= &
 f(x)+\langle \nabla f(x),q \rangle_p+D_h(x+q,x)_{pL}
\end{eqnarray*}
and thus $(f,\hat{S})\sim \mathrm{ESO}_h(L\ones)$. In other words, if $f$ is $L$--smooth relative to separable function $h$, then $(f,\hat{S})\sim \mathrm{ESO}_h(L\ones)$ with any sampling $\hat{S}$.

However, when considering a specific sampling strategy, it might be possible to choose smaller ESO parameters $v$, allowing us to obtain a faster convergence comparing to the deterministic method using full gradient in each iteration. As we show later, if ESO parameters are chosen to be smoothness parameters, we do not obtain any speedup comparing the deterministic method.

There are various examples of functions satisfying $h$--ESO:
\begin{itemize}
\item If $h(x)=\|x\|^2/2$, definition of $h$--ESO matches definition of standard Expected Separable Overapproximation introduced in \cite{ESO}, \cite{pcd_takac}. Under the assumption that $f$ is $A^\top A$ smooth, i.e. $\forall x,y$:
\begin{equation*}
f(y)
\oo{\leq}
 f(x)+\langle \nabla f(x),y-x\rangle+\frac{1}{2}(y-x)^\top A^\top A (y-x),
\end{equation*}
one can prove that $(f,\hat{S})\sim \mathrm{ESO}_h(v)$ if 

$$
P(\hat{S}) \circ (A^{\top}A)
\oo{\preceq}
 \mathrm{Diag}\left(p(\hat{S})\circ v\right),
$$

where $P(\hat{S})$ and $p(\hat{S})$ are respectively probability matrix and probability vector of sampling $\hat{S}$. 

Note that $A^\top A$ smoothness is equivalent to relative smoothness for $L=1,\,h(x)=\tfrac12 x^{\top}A^{\top}Ax$ and arises naturally if the objective $f$ is in the form 
\[
f(x)
\oo{\eqdef}
\sum_{i=1}^n \phi^{(i)}\left(M_{(i)}x\right),
\]
where function $\phi^{(i)}$ is $\gamma^{(i)}$ smooth. In this case, $f$ is $A^\top A\eqdef \sum_{i=1}^n \gamma^{(i)}M_{(i)}^{\top} M_{(i)}$ smooth. As an example, for uniform sampling (when every iteration is only one coordinate sampled uniformly at random), $v$ can be chosen as $\mathrm{Diag}\left(A^{\top}A\right)$. In contrast, the tightest smoothness parameter that can be chosen here is the maximal eigenvalue of $\left(A^{\top}A\right)$, which is in general even greater than maximal diagonal element of $\left(A^{\top}A\right)$. 

For more details about how to choose $v$ for arbitrary sampling $\hat{S}$ or proofs of the statements above, see \cite{ESO}.
\item D-optimal design problem.
\begin{eqnarray*}
\min_x  && f(x) \eqdef − \log \det\left(H\,\mathrm{Diag}(x)H^\top\right)
\\
\text{subject to} && \langle \ones, x\rangle  = 1
\\
&& x \in \R^n_+,
\end{eqnarray*}
where matrix $H\in \R^{m\times n}$ has rank $n$, $n \geq m + 1$.
In this case f is 1 relative smooth with respect to $h(x)\eqdef-\sum_{i=1}^n \log\left(x^{(i)}\right)$ \cite{Rel_smoothness_nesterov}. Thus, function $(f,\hat{S})\sim \mathrm{ESO}_h(\ones)$ for any sampling $\hat{S}$.   

\item Poisson linear inverse problem. 
The task here is to find vector $x\in \R^n_+$ to minimize $\mathrm{KL}(Ax\|b)$ for matrix $A\in \R^{m\times n}_+$ and vector $b \in \R^n_+$. 

Thus optimization problem here is the following:
\begin{eqnarray*}
\min_x  && f(x) \eqdef \sum_{i=1}^m f^{(i)}(x)=\sum_{i=1}^n\Big( b^{(i)} \log \frac{b^{(i)}}{(Ax)^{(i)}}+(Ax)^{(i)}-b^{(i)} \Big)
\\
\text{subject to} &&  x \in  \R^n_+.
\end{eqnarray*}

Again, in this case $f$ is $\sum_{i=1}^m b^{(i)}$--smooth relative to $h(x)\eqdef-\sum_{i=1}^n \log(x^{(i)})$ \cite{bauschke2016NoLips}. Thus, as before, $(f,\hat{S})\sim \mathrm{ESO}_h(\ones \sum_{i=1}^m b^{(i)})$ for any sampling $\hat{S}$. 

Considering regularized Poisson linear inverse problem:
$$
\min_x \quad f(x)\oo{\eqdef}\mathrm{KL}(Ax\|b)+\mu r(x)
$$
with logarithmic regularizer $r(x)=-\sum_{i=1}^m \log(x^{(i)})$. Then we have \[(f,\hat{S})\sim \mathrm{ESO}_h\left(\left(\sum_{i=1}^m b^{(i)}+ \mu\right)\ones\right)\]
for any sampling $\hat{S}$.

\end{itemize}

The following lemma gives an example on function $f$ which is $h$--ESO where $h$ is not $\frac12 \|x\|^2$ with parameters $v$ potentially $n$ times smaller than relative smoothness constant $L$.

\begin{lemma}
Suppose that
\[
f(x)\oo{\eqdef}f_1(x)+f_2(x),
\]
where $f_1$ is $L_1$ smooth relative to $h_1$ and $f_2$ is $A^{\top}A$ smooth ($1$--smooth relative to $h_2(x)\eqdef\tfrac{1}{2}x^{\top}A^{\top}Ax$). 
Let us consider $\hat{S}$ to be uniform sampling which samples a single coordiante (uniformly) each iteration. 
Then, \[(f,\hat{S})\sim \mathrm{ESO}_h\left(\max\left(L_1\ones,\mathrm{diag}(A^{\top}A)\right)\right)\]  for $$h(x)\oo{\eqdef}h_1(x)+\frac12\|x\|^2.$$
\label{l:eso_example}
\end{lemma}
\begin{proof}
From ESO theory in standard smooth setting we have that $(f_2,\hat{S})\sim \mathrm{ESO}_{h_2} \left(\mathrm{diag}\left(A^{\top}A\right)\right)$ for $h_2(x)\eqdef\frac12\|x\|^2$. Clearly, $(f_1,\hat{S})\sim \mathrm{ESO}_{h_1}(L_1\ones)$. Summing ESO inequality for $f_1$ and $f_2$ we get

\begin{eqnarray*}
\Exp{f\left(x+\sum_{i\in \hat{S}}q^{(i)}\ones^i\right)}
&=&
\Exp{f_1\left(x+\sum_{i\in \hat{S}}q^{(i)}\ones^i\right)
+
f_2\left(x+\sum_{i\in \hat{S}}q^{(i)}\ones^i\right)}
\\
&\leq&
 f_1(x)+\langle \nabla f_1(x),q \rangle_p+D_{h_1}(x+q,x)_{p\circ L_1\ones}
 \\
 && \qquad +
  f_2(x)+\langle \nabla f_2(x),q \rangle_p+\frac12\|q\|^2_{p\circ v_2}
  \\
 &\leq&
 f(x)+\langle \nabla f(x),q \rangle_p+D_{h}(x+q,x)_{p\circ v},
\end{eqnarray*}
which concludes the proof.

\end{proof}

Note that in the lemma above $f$ is $L$--smooth relative to $h$ where $L\eqdef\mathrm{max}\left(L_1,\lambda_{\mathrm{max}}(A^{\top}A)\right)$, and in general one cannot find tighter constant $L$.  Clearly, $v^{(i)}\leq L$ for all $i$ and in the case when $A=\ones$ and $L_1<1$ we have $\lambda_{\mathrm{max}}(A^{\top}A)=n$ and thus $L=n$, in contrast to $v=\ones$, thus $L$ might be $n$ times larger than ESO parameters $v$. We also note without the proof that it is possible to design ESO parameters for block coordinate descent as well analogously. 

\subsection{Algorithm}
Let us now proceed with the algorithm. We introduce here Relative Randomized Coordinate Descent (relRCD) - algorithm for minimizing functions satisfying Relative ESO assumption. The main idea is very simple - each iteration sample a subset of coordinates with respect to sampling $\hat{S}$ and update them according to Relative ESO assumption. We only consider sampling strategies such that all coordinates have equal chance to be sampled. 

\begin{algorithm}[H]
\textbf{Input: }{Initial iterate $x_0$, separable reference function $h$, positive vector $v$ and sampling $\hat{S}$ $(f,\hat{S})\sim \mathrm{ESO}_h(v)$ and $\Prob(i\in \hat{S})=\Prob(j\in \hat{S})$ for all $i,j\leq n$.}\\
\For {$t= 0,1,\dots, k-1$} {
	 \begin{enumerate}
\item Choose randomly $M_t\in \{1,2,\dots m\}$ according to the sampling $\hat{S}$  
\item Set $Q_t\leftarrow\left\{x \;\Big |\;x=x_t+\sum_{i\in M_t}\mathrm{span}\left(\ones^i\right) \right\}$ 
\item Set $x_{t+1}\leftarrow \argmin_{x\in Q_t}\langle\nabla f(x_t),x\rangle+D_h(x,x_t)_v$ 
\end{enumerate}
 }
\textbf{return} $x_k$
\caption{relRCD (Relative Randomized Coordinate Descent)}
\end{algorithm}

\subsection{Analysis}

First of all, we introduce the variant of three point property, which we will use later in the analysis. 
For simplicity we denote probability vector of sampling $\hat{S}$ as $p$ throughout this section. Since all coordinates have the same probability to be sampled, we can write $p=p_0\ones$ for some scalar $p_0$ such that $0<p_0\leq 1$. 

\begin{lemma}[Three point property for ESO]

Let $c,\, v,\,p \in \R^n$ and $D_h(\cdot,\cdot)$ be a Bregman distance for separable function $h(x)=\sum_{i=1}^nh^{(i)}(x)$, both defined on some arbitrary set $Q$. For a given $z \in Q$ denote
$$z_+\oo{\eqdef}\argmin_{x\in Q} \left\{\langle c,x \rangle_p+D_{h}(x,z)_{p\circ v}\right\},$$
where 
$
D_{h}(x,z)_{p\circ v}
= 
\sum_{i=1}^n D_{h^{(i)}}\left(x^{(i)},z^{(i)}\right)
p^{(i)}v^{(i)}
$.
Then for all $x\in Q$ we have
\begin{equation}
\langle c,x \rangle_p+D_{h}(x,z)_{p\circ v}
\oo{\geq} 
\langle c,z_+ \rangle_p+D_{h}(z_+,z)_{p\circ v}+D_{h}(x,z_+)_{p\circ v} .
\label{eq:tpp_eso}
\end{equation}
\end{lemma}

\begin{proof}
Define $c'=c\circ p$ and $h'(x)=\sum_{i=1}^np^{(i)}v^{(i)}h^{(i)}\left(x^{(i)}\right)$. Thus we have
\[ 
z_+\oo{=}\argmin_{x\in Q} \{\langle c',x \rangle+D_{h'}(x,z)\}.
\]
It remains to apply  the three point property (Lemma~\ref{L:tpp}).

\end{proof}

The next lemma provides us with the expected decrease in objective for each iteration of Algorithm~\algE\, and has the same role as Lemma~\ref{iter_decrease_mb} in the analysis if Algorithm~\algR\,.

For notational simplicity, denote throughout this section

\begin{equation}
x_{(t+1,*)}\oo{\eqdef} \mathrm{argmin}_{x\in Q} {\langle\nabla f(x_t)}\ x\rangle+D_h(x,x_t)_v.
\label{eq:t+1_star}
\end{equation}

\begin{lemma}[Iteration decrease for Algorithm~\algE]
Suppose that $f$ is $w$--Relative Strongly Convex with respect to $h$ and $(f,\hat{S})\sim \mathrm{ESO}_h(v)$ for $p(\hat{S})=p=p_0\ones$. Denote $\Delta=\mathrm{min} \frac{w^{(i)}}{v^{(i)}} $. Then, one iteration of relRCD satisfies
\begin{equation*}
\Exp{f(x_{t+1})} 
\oo{\leq} 
(1-p_0)\Exp{f(x_{t})} +p_0f(x_*) +(1-p_0\Delta)\Exp{D_h(x_*,x_t)_{ v}}-\Exp{D_h(x_*,x_{t+1})_v}.
\end{equation*}
\label{L: eso^{(i)}teration}
\end{lemma}

\begin{proof}
Let us write $h$--ESO for $x=x_t$, $q=x_{(t+1,*)}-x_t$ and sampling $\hat{S}$. We get

\begin{eqnarray}
\Exp{f(x_{t+1})\;|\;x_t} 
&\stackrel{\eqref{eq:eso_def}}{\leq} &
 f(x_{t}) + \langle \nabla f(x_t) ,x_{(t+1,*)}-x_t \rangle_p+D_{h}(x_{(t+1,*)},x_t)_{p\circ v}
 \nonumber
\\
&\stackrel{\eqref{eq:tpp_eso}}{\leq} &
 f(x_{t}) + \langle \nabla f(x_t) ,x-x_t \rangle_p+D_{h}(x,x_t)_{p\circ v}-D_{h}(x,x_{(t+1,*)})_{p\circ v}
  \nonumber
\\
&\stackrel{\eqref{eq:sc_vec_def}}{\leq}  &  
(1-p_0)f(x_{t}) +p_0f(x) -D_h(x,x_t)_{p\circ w}+D_{h}(x,x_t)_{p\circ v}-D_{h}(x,x_{(t+1,*)})_{p\circ v}
 \nonumber
\\
&\leq & 
(1-p_0)f(x_{t}) +p_0f(x) +(1-\Delta)D_h(x,x_t)_{p\circ v}-D_{h}(x,x_{(t+1,*)})_{p\circ v}.
\label{eq:nonfinished_eso}
\end{eqnarray}
In the last inequality above we used the definition of $\Delta$.
Since 
\[\Exp{D_h(x,x_{t+1})_v\;|\;x_t}
\oo{=}
(1-p_0)D_h(x,x_t)_v+p_0D_h(x,x_{(t+1,*)})_v,\] 
we have 
\[D_h(x,x_{(t+1,*)})_{p\circ v}
\oo{=}
\Exp{D_h(x,x_{t+1})_v\;|\;x_t} -(1-p_0)D_h(x,x_t)_v
.\]
 Plugging it back to \eqref{eq:nonfinished_eso} we obtain

\begin{eqnarray}
\Exp{f(x_{t+1})\;|\;x_t} 
&\leq &
(1-p_0)f(x_{t}) +p_0f(x) +(1-\Delta)D_h(x,x_t)_{p\circ v}-\Exp{D_h(x,x_{t+1})\;|\;x_t} 
\nonumber
\\
&& \qquad
+(1-p_0)D_h(x,x_t)_v
\nonumber
\\
\nonumber
&=&
(1-p_0)f(x_{t}) +p_0f(x) +(1-p_0\Delta)D_h(x,x_t)_{ v}-\Exp{D_h(x,x_{t+1})_v\;|\;x_t}.
\\
\label{eso_2}
\end{eqnarray}
Taking the expectation over the algorithm and using the tower property we obtain the desired result. 
\end{proof}

Now, we are ready to introduce first of the two main results of this work - Theorems~\ref{theorem_rcd_eso} and~\ref{th:rcd_eso_nonsc}, providing with a convergence rate of relRCD under ESO assumption.

\subsubsection{Strongly convex case $w\in \R^n_{++}$}

\begin{theorem}[Convergence rate for Algorithm~\algE ]\label{theorem_rcd_eso}
Suppose that $f$ is $w$--strongly convex relative to $h$ for $w\in \R^n_+$ and that $(f,\hat{S})\sim \mathrm{ESO}_h(v)$ for $p(\hat{S})=p=p_0\ones$. Denote $\Delta=\mathrm{min} \frac{w^{(i)}}{v^{(i)}}$. Then, iterates of Algorithm~\algE\ satisfy:
\begin{equation}
\sum_{t=1}^{k}c_t \big(\Exp{f(x_t)}-f(x_*)\big)
\oo{\leq} 
\frac
{(1-p_0\Delta)D_{h}(x_*,x_0)_v  +(1-p_0)(f(x_0)-f(x_*)) }
{1-\Delta^{-1} + \Delta^{-1} \Big ( \frac{1}{1-p_0\Delta} \Big )^{k-1} }, 
\label{eq:eso_main}
\end{equation}
where $c\in \R^k$ is a positive vector with entries summing up to 1.
On top of that, we have
\begin{equation}
\Exp{D_h(x_*,x_k)_v}
\oo{\leq}
 (1-p_0\Delta)^kD_h(x_*,x_0)_{ v},
\label{eq:eso_breg^{(i)}terates}
\end{equation}
and
\begin{equation}
\frac 1k\sum_{t=1}^{k}\Exp{D_{h}(x_{t},x_{(t+1,*)})_v}
\oo{\leq}
 \frac{f(x_0)-f(x_{*})}{kp_0}.
\label{eq:eso_norm_grad}
\end{equation}
\end{theorem}

\begin{proof}
The proof of \eqref{eq:eso_main} follows by applying Lemma~\ref{l:rate_from_iter} together with Lemma~\ref{L: eso^{(i)}teration} for 
$f_t=\Exp{f(x_t)},\,D_t=\Exp{D_h(x_*,x_t)_v},\,f_*= f(x_*),\, \delta=p_0,\,\ \varphi=1,\, \psi=\Delta$.

Inequality \eqref{eq:eso_breg^{(i)}terates} follows recursively from  
\begin{equation*}
\Exp{D_h(x_*,x_{t+1})_v}  \oo{\leq} 
(1-p_0\Delta)\Exp{D_h(x_*,x_t)_{ v}},
\end{equation*}
which holds due to Lemma~\ref{L: eso^{(i)}teration} as $\Exp{f(x_t)}$ is an nonincreasing sequence.   

Finally, to prove inequality \eqref{eq:eso_norm_grad} let us set $x=x_t$ in \eqref{eso_2} to obtain 
\[
\Exp{f(x_{t+1})\;|\;x_t} 
\oo{\leq} 
f(x_{t})-\Exp{D_h(x_t,x_{t+1})_v\;|\;x_t}
.\]
Taking the full expectation, averaging over iterations and using $\Exp{f(x_k)}\leq f(x_*)$ we get
\[
\frac 1k\sum_{t=1}^{k}\Exp{\Exp{D_{h}(x_{t},x_{t+1})_v\;|\;x_t}}
\oo{\leq}
 \frac{f(x_0)-f(x_{*})}{k} 
.\]
It remains to notice that $\Exp{D(x_t,x_{t+1})\;|\;x_t}=p_0D(x_t,x_{(t+1,*)})$.
\end{proof}

Convergence rates from \eqref{eq:eso_main} and \eqref{eq:eso_breg^{(i)}terates} are both asymptotically driven by the term 
\[
(1-p_0\Delta)^k 
\oo{=}
\left(1-p_0\min_i \frac{v^{(i)}}{w^{(i)}}\right)^k.
\]

Therefore, no speedup is obtained comparing to relGD (Theorem~\ref{t:primal_scheme}), for ESO parameters set as $v=L\ones$ and strong convexity parameters set as $w=\mu\ones$. However, if one set ESO parameters $v$ more tightly, taking into the consideration the specific probability sampling, one can outperform Algorithm~\algD\,. There is a broad theory about how to compute ESO parameters $v$ for various different sampling strategies in case of $h(x)=\frac12\|x\|^2$, see \cite{ESO}. We gave the example of one class of functions in Lemma~\ref{l:eso_example}. 

Note also that \eqref{eq:eso_main} provides an asymptotically same convergence result as Randomized Coordinate Descent in the standard smooth settingfor uniform sampling \cite{NSync}, therefore we obtained a good generalization in this case.

To conclude this section, notice that \eqref{eq:eso_norm_grad} provides us with a convergence of $\Exp{D_h(x_t,x_{t+1})_v}$. Quantity $D_{h}(x_{t},x_{t+1})_v$ depends on $x_t$, $h$ and $f$ and goes to 0 when $\nabla f(x_t)$ goes to 0 (this can be easily seen from \eqref{eq:t+1_star}). Thus $D_{h}(x_{t},x_{t+1})_v$ can be considered as a ``norm'' of $\nabla f(x_t)$ which depends on $x_t$ and $h$. In the standard setting when $h(x)=\| x\|^2/2$ and $v=L\ones$ we have
\[
D_{h}(x_{t},x_{t+1})_v
\oo{=}
LD_{h}(x_{t},x_{t+1})
\oo{=}
L\left \|\frac{1}{L}\nabla f(x_k)\right\|^2
\oo{=}
\frac1L\|\nabla f(x_k)\|^2,
\]
and thus we obtain the convergence of the norm of gradient in this case.

\begin{remark}
According to Theorem~\ref{theorem_rcd_eso}, one needs 
\begin{equation}
\frac{\Delta}{p_0} \log(\Delta)\log\left( \frac{(1-p_0\Delta)D_{h}(x_*,x_0)_v  +(1-p_0)(f(x_0)-f(x_*)) }{\epsilon}+\Delta^{-1}-1 \right)
\label{eq:rcd_eso_epsilon_comp}
\end{equation}
iterations for Algorithm~\algE\ to converge to $\epsilon$-optimality in functional values and 
\begin{equation}
\frac{\Delta}{p_0}\log\left(\frac{D_{h}(x_*,x_0)_v }{\epsilon} \right)
\label{eq:rcd_eso_epsilon_comp_dist}
\end{equation}
iterations to get to $\epsilon$-neighborhood to the optimum in (Bregman) distance. For a comparison, randomized coordinate descent in standard smooth setting requires 
\[
\frac{\Delta}{p_0}\log\left(\frac{f(x_0)-f(x_*)}{\epsilon} \right)
\]
iterations to reach to $\epsilon$-optimality, which is essentially same as both~\eqref{eq:rcd_eso_epsilon_comp} and~\eqref{eq:rcd_eso_epsilon_comp_dist}.
\end{remark}

\subsubsection{Non-strongly convex case: $\min w_{i}=0$}

The following theorem provides us with the convergence rate of Algorithm~\algE\, when $f$ is convex but not necessarily relative strongly convex (i.e., $\min w_i=0$).

\begin{theorem}[Convergence rate for Algorithm~\algE ]
Suppose that $f$ is convex and $(f,\hat{S})\sim \mathrm{ESO}_h(v)$ for $p(\hat{S})=p=p_0\ones$ and separable convex function $h$. Running Algorithm~\algE\ for $k$ iterations we obtain:
\begin{equation*}
\sum_{t=1}^{k}c_t(\Exp{f(x_t)}-f(x_*)) 
\oo{\leq} 
\frac{D_{h}(x,x_0)_v+(1-p_0)\left(f(x_0)-f(x_*)\right)}{1+p_0(k-1)},
\end{equation*}
where $c=(c_1,\dots,c_k)\in \R^k$ is a positive vector proportional to 
$
\left(p_0,\,p_0\,\dots,\,p_0,\,1 \right)
$.
\label{th:rcd_eso_nonsc}
\end{theorem}

\begin{proof}
For simplicity, denote $r_t=\Exp{f(x_t)}-f(x_*)$.
We can follow the proof of Theorem~\ref{theorem_rcd_mb} using Lemma~\ref{l:rate_from_iter} to get the equation \eqref{RCD_tel}, which can be rewritten for $\mu=0$ as follows:
\begin{equation*}
D_{h}(x,x_0)_v
\oo{\geq} 
r_{k}+
p_0\sum_{t=1}^{k-1}  r_t -(1-p_0)r_0,
\end{equation*}
which can be easily rearranged as

\begin{equation*}
\frac{D_{h}(x,x_0)_v+(1-p_0)r_0}{1+(k-1)p_0}
\oo{\geq}  
\frac{1}{1+(k-1)p_0}\left(r_{k}+
p_0\sum_{t=1}^{k-1}  r_t \right).
\end{equation*}
\end{proof}

As previously, Theorem~\ref{theorem_rcd_nsc_mb} captures known results of Relative Gradient Descent for $p_0=1$ (Theorem~\ref{t:primal_scheme}).

 \section{Relative Stochastic Gradient Descent\label{S:sgd}}

In this section, we assume that every iteration we have an access to the stochastic oracle providing us $ \widetilde{g}_t$ -- an unbiased estimator of $\nabla f(x_t)$. The next iterate of the algorithm is obtained using the stochastic gradient instead of the true gradient. The analogous algorithm in the standard smooth setting is Stochastic Gradient Descent which is in fact a special case of Relative Stochastic Gradient Descent. 

\subsection{Algorithm}

The iterates of standard stochastic gradient descent with stepsize sequence $\{\gamma_t \}_{t=0}^\infty$ are the following
\begin{equation}
x_{t+1}\oo{\leftarrow} x_t-\gamma_t \widetilde{g_t}.
\label{eq:sgd_standard}
\end{equation}

It is known that unlike gradient descent, scheme \eqref{eq:sgd_standard} does not necessarily guarantee the convergence to the optimum, as the variance of gradient estimator $\widetilde{g_t}$ might not converge to zero, resulting in the convergence to the neighborhood of the optimum. This is where the importance of decreasing stepsize sequence $\{\gamma_t \}_{t=0}^\infty$  takes a place; thus taking more conservative steps as progressing with the algorithm. However, in particular special cases, such as empirical risk minimization, a different tricks tricks guarantee vanishing variance of gradient estimator as one approach optimum \cite{sag, SAGA, svrg, SDCA, sarah}.

In this work, we attain the convergence of Relative Stochastic Gradient Descent by making the algorithm conservative over time. We leave the variance reduction for relatively smooth ERM problems as an open research question.

\begin{algorithm}[H]
\textbf{Input: }{Initial iterate $x_0$, separable reference function $h$, positive scalar $L$ such that  $f$ is $L$-relative smooth with respect to $h$, stepsize determining sequence $\{L_t \}_{t=0}^{\infty}$ with $L_0=L$.}\\
\For {$t= 0,1,\dots, k-1$} {
	 \begin{enumerate}
\item Get   $\widetilde{g}_t$ such that $\Exp{\widetilde{g}_t}=\nabla f(x_t)$
\item Set $x_{t+1}\leftarrow \argmin_{x\in Q} \left\{\langle \widetilde{g}_t,x\rangle+L_tD_h(x,x_t)\right\}$ 
\end{enumerate}
 }
\textbf{return} $x_k$
\caption{relSGD (Relative Stochastic Gradient Descent)}
\end{algorithm}

Recall that for the special choice $D_h(x,y)=\frac12\Vert x-y\Vert ^2$, Stochastic Gradient Descent with nonincreasing stepsize $\gamma_t=\frac1{L_t}$ is recovered. Define the new iterate using the true gradient as
\[
x_{(t+1,*)}
\oo{\eqdef}
\argmin_{x\in Q}\left\{\langle\nabla f(x_t),x\rangle+L_tD_h(x,x_t)\right\},
\]
which will be used only in Assumption~\ref{a:sgd_variance}, and will never be evaluated in the actual run of the algorithm. 

Throughout this section, we will make the following assumption,  which is in fact closely related to boundedness of variance of the gradient estimator, as Remark~\ref{rem:bound_var} shows. Notice that boundedness of variance of gradient estimator is very common in SGD literature. 

\begin{assumption}
\label{a:sgd_variance}
There exist  $\sigma\neq 0$ such that for all $t$ we have

\begin{equation}
L_t\Exp{\left\langle \nabla f(x_t)- \widetilde{g}_t, x_{t+1} -x_{(t+1,*)}\right \rangle \;|\;x_t}\oo{\leq} \sigma^2.
\label{eq:sgd_variance_assump}
\end{equation}

\end{assumption}

\begin{remark}
\label{rem:bound_var} Consider Assumption~\ref{a:sgd_variance}. 
If we additionally assume that $h$ is $\mu_h$-strongly convex function, we obtain

\begin{eqnarray*}
L_t\Exp{\langle \nabla f(x_t)- \widetilde{g}_t, x_{t+1} -x_{(t+1,*)}\rangle\;|\;x_t}
&\leq&
L_t\Exp{\big\Vert  \nabla f(x_t)- \widetilde{g}_t\big\Vert  \cdot \big\Vert  x_{t+1} -x_{(t+1,*)}\big\Vert \;\Big|\;x_t}
\\
&\stackrel{(*)}{\leq}&
L_t\Exp{\big\Vert  \nabla f(x_t)- \widetilde{g}_t\big\Vert  \cdot \frac{1}{\mu_h} \left \Vert  \frac1{L_t} ( \widetilde{g}_t-\nabla f(x_t)) \right \Vert \;\bigg|\;x_t}
\\
&=&
\frac{1}{\mu_h}
\Exp{\big\Vert  \nabla f(x_t)- \widetilde{g}_t\big\Vert ^2 \;\Big|\;x_t}.
\end{eqnarray*}

Inequality $(*)$ holds due to $\mu_h$--strong convexity of $h$, since 
$$
\Vert  x_{t+1}-x_{(t+1,*)}\Vert 
\oo{\leq}
\frac{1}{\mu_h}
\Vert \nabla h(x_{t+1})-\nabla h(x_{(t+1,*)})\Vert 
\oo{=}
\frac{1}{\mu_h}
\left \Vert  \frac{1}{L_t} ( \widetilde{g}_t-\nabla f(x_t))\right\Vert. 
$$ 

Thus, if $h$ is $\mu_h$--strongly convex, $\sigma^2$ can be chosen so that $\sigma^2 \mu_h$ correspond to the global upper bound on variance of gradient estimator. 
\end{remark}

\subsection{Key Lemma}

The following lemma is key for this section and provides us with a bound on expected suboptimality in iteration $t$.

\begin{lemma}[Iteration decrease for Algorithm \algS]
Suppose that $f$ is $L$--smooth and $\mu$--strongly convex relative  to function $h$. Performing one iteration of Algorithm~\algS\ we obtain for all $x\in Q$
\begin{eqnarray}
\Exp{f(x_{t+1})\;|\;x_t}-f(x)
&\leq& 
(L_t-\mu)D_h(x,x_t)- L_t\Exp{D_h(x,x_{t+1})\;|\;x_t}
\nonumber
\\
&& \qquad \qquad +\frac{\sigma^2}{L_t}-(L_t-L)\Exp{D_h(x_{t+1},x_t) \;|\;x_t}.
\label{Eq:sgd_dec^{(i)}ter_general}
\end{eqnarray}

\label{L:sgd^{(i)}ter_decrease_general}
\end{lemma}

\begin{proof}

\begin{eqnarray*}
\Exp{f(x_{t+1})\;|\;x_t}
&\stackrel{\eqref{eq:relsmooth}}{\leq}&
f(x_{t})+ \Exp{\langle\nabla f(x_t), x_{t+1}-x_t \rangle +LD_h(x_{t+1},x_t)\;|\;x_t}
\\
&=&
 f(x_{t})+ \Exp{\langle\nabla f(x_t), x_{t+1}-x_t \rangle +L_tD_h(x_{t+1},x_t) }
 \\
 & & \qquad -(L_t-L)\Exp{D_h(x_{t+1},x_t)\;|\;x_t}
 \\
&=&
 f(x_{t})+ \Exp{\langle\nabla f(x_t), x_{t+1}-x_t \rangle +L_tD_h(x_{t+1},x_t)\;|\;x_t}
 \\
 && \qquad-(L_t-L)\Exp{D_h(x_{t+1},x_t)\;|\;x_t}
 \\
&=&
f(x_{t})+ \Exp{\langle \widetilde{g}_t, x_{t+1}-x_t \rangle +L_tD_h(x_{t+1},x_t)\;|\;x_t}
\\
&&\qquad
+\Exp{\langle\nabla f(x_t), x_{t+1}-x_t \rangle-\langle \widetilde{g}_t, x_{t+1}-x_t \rangle\;|\;x_t}
\\ && \qquad
-(L_t-L)\Exp{D_h(x_{t+1},x_t)\;|\;x_t}
\\
&\stackrel{\eqref{eq:tpp}}{\leq}&
 f(x_{t})+ \Exp{\langle \widetilde{g}_t, x-x_t \rangle +L_tD_h(x,x_t)-L_tD_h(x,x_{t+1})\;|\;x_t}
 \\
 &&
\qquad 
 +\Exp{\langle \nabla f(x_t)- \widetilde{g}_t, x_{t+1}-x_t \rangle\;|\;x_t}
 -(L_t-L)\Exp{D_h(x_{t+1},x_t)\;|\;x_t}
 \\
 &=&
 f(x_{t})+ \langle \nabla f(x_t), x-x_t \rangle +L_tD_h(x,x_t)-L_t\Exp{D_h(x,x_{t+1})\;|\;x_t}
 \\
& & \qquad +
\Exp{\langle \nabla f(x_t)- \widetilde{g}_t, x_{t+1}-x_t \rangle\;|\;x_t}
-(L_t-L)\Exp{D_h(x_{t+1},x_t)\;|\;x_t}
 \\
&\stackrel{\eqref{eq:relsc}}{\leq}&
f(x)+(L_t-\mu)D_h(x,x_t)- L_t\Exp{D_h(x,x_{t+1})\;|\;x_t}
\\
&& \qquad +
\Exp{\langle \nabla f(x_t)- \widetilde{g}_t, x_{t+1} -x_t \rangle\;|\;x_t}
-(L_t-L)\Exp{D_h(x_{t+1},x_t)\;|\;x_t}
 \\
&\stackrel{(*)}{=}&
f(x)+(L_t-\mu)D_h(x,x_t)- L_t\Exp{D_h(x,x_{t+1})\;|\;x_t}
\\
&&\qquad +
\Exp{\langle \nabla f(x_t)- \widetilde{g}_t, x_{t+1} -x_{(t+1,*)} \rangle\;|\;x_t}
-(L_t-L)\Exp{D_h(x_{t+1},x_t)\;|\;x_t}
\\
&\stackrel{\eqref{eq:sgd_variance_assump}}{\leq}&
f(x)+(L_t-\mu)D_h(x,x_t)- L_t\Exp{D_h(x,x_{t+1})\;|\;x_t}
+\frac{\sigma^2}{L_t}
\\
&&\qquad -
(L_t-L)\Exp{D_h(x_{t+1},x_t)\;|\;x_t}.
\end{eqnarray*}
Equality (*) follows from fact that $\widetilde{g}_t$ is unbiased and thus we have
\[
\Exp{\langle \nabla f(x_t)- \widetilde{g}_t, x_t \rangle\;|\;x_t}
\oo{=}
\Exp{\langle \nabla f(x_t)- \widetilde{g}_t, x_{(t+1,*)} \rangle\;|\;x_t}
\oo{=}
0.\] 
\end{proof}

Note that Lemma~\ref{L:sgd^{(i)}ter_decrease_general} is very similar to Lemma~\ref{iter_decrease_mb} for $\mb=n$. There are only two additional terms in \eqref{Eq:sgd_dec^{(i)}ter_general} -- $\tfrac{\sigma^2}{L_t}$ appears due to the noise in the gradient estimator and  
$(L_t-L)\Exp{D_h(x_{t+1},x_t)|x_t}$ appears due to the varying stepsize rule. We now derive the convergence rate of relSGD for various stepsize rules.

\subsection{Constant stepsize rule}

The following theorem provides a convergence result of SGD with constant stepsize rule using recursively Lemma~\ref{L:sgd^{(i)}ter_decrease_general} -- it shows that Relative Stochastic Gradient Descent converges linearly to a particular neighborhood of the optimum. We mention it for completeness, to illustrate that relSGD in our fully general relative smooth setting behaves very similar to standard (smooth) SGD.

\begin{theorem}[Constant stepsize rule for Algorithm~\algS]
Suppose that $f$ is $L$--smooth and $\mu$--strongly convex relative to $h$. Iterates of Algorithm~\algS\ with stepsize rule $L_t=L$ satisfy:
\begin{equation}
\sum_{t=1}^{k}c_t \left (\Exp{f\left (x_t\right )}-\left (f\left (x_*\right )+\frac{\sigma^2}{L}\right )\right )
\oo{\leq} 
\frac{ D_{h}\left (x_*,x_0\right ) \mu }{ \left ( \frac{L}{L-\mu }  \right )^{k}-1}, 
\label{eq:sgd_const}
\end{equation}
where $c$ is positive vector proportional to 
$\left (1,\, \beta,\, \beta^2,\dots,\, \beta^{k-1}\right )$
 summing up to 1 for 
 $$\beta\oo{\eqdef}\frac{L}{L-\frac{\mu}{m}}.$$
\end{theorem}
 
\begin{proof}

Let us set $x=x_*$ in \eqref{Eq:sgd_dec^{(i)}ter_general}, take the expectation of over the algorithm and use the tower property. We obtain
\begin{equation*}
\Exp{f(x_{t+1})}-\left( f(x_*)+\frac{\sigma ^2}{L} \right )
\oo{\leq}  
(L-\mu)\Exp{D_h(x_*,x_t)}- L\Exp{D_h(x_*,x_{t+1})}.
\end{equation*}
The proof now follows directly by applying Lemma~\ref{l:rate_from_iter} the inequality above for
$f_t=\Exp{f(x_t)},\,D_t=\Exp{D_h(x_*,x_t)_v},\, f_*= f(x_*)+\tfrac{\sigma^2}{L},\, \delta=1,\,\ \varphi=L,\, \psi=\mu$.

\end{proof}

Inequality \eqref{eq:sgd_const} shows that the sequence of iterates $\{x_t\}$ converges linearly to the set $\{x:f(x)\leq f(x_*)+\frac{\sigma^2}{L}\}$, and the convergence rate is driven by the term
$
\left(1-\frac{\mu}{L}\right)^k.
$

\subsection{Decreasing stepsize rule}
The following theorem is one of two key results of this work, together with Theorem~\ref{theorem_rcd_eso}. It provides us with a convergence result of Algorithm~\algS\ for a general stepsize rule.

\begin{theorem}[General convergence for Algorithm~\algS]
Suppose that $f$ is $L$--smooth and $\mu$--strongly convex relative to $h$. Define $c_0=1$ and $c_t=\frac{L_{t-1}}{L_{t}-\mu}c_{t-1}$ for $t\geq 1$ and $C_k=\sum_{t=1}^k c_{t-1}$. Then, Algorithm~\algS\ satisfies:
\begin{equation}
\sum_{t=1}^k \frac{c_{t-1}}{C_k} \Exp{f(x_t)-f(x_*)}\oo{\leq} 
\frac{(L-\mu)D_h(x_*,x_0)}{C_k}+ \sigma^2\sum_{t=0}^{k-1}  \frac{c_t}{C_kL_t} .
\label{sgd_dec_Cbound}
\end{equation}
\label{t: SGD_Ctheorem}
\end{theorem}

\begin{proof}
Let us set $x=x_t$ in \eqref{Eq:sgd_dec^{(i)}ter_general}, take 
the expectation of over the algorithm and use tower property. Ignoring the last term we get
\begin{equation*}
\Exp{f(x_{t+1})-f(x_*)}
\oo{\leq} 
(L_t-\mu)D_h(x_*,x_t)- L_t\Exp{D_h(x_*,x_{t+1})}+\frac{\sigma^2}{L_t}.
\end{equation*}
Multiplying the above by by $c_t$ and summing for $t=0$ to $k-1$ we obtain 

\begin{eqnarray*}
\sum_{t=1}^k c_{t-1} \Exp{f(x_t)-f(x_*)}
&\leq &
 (L_0-\mu)D_h(x_*,x_0)-c_{k-1}L_k\Exp{D_h(x_*,x_k)}+ \sigma^2\sum_{t=0}^{k-1}  \frac{c_t}{L_t}  
 \\
 &\leq &
 (L_0-\mu)D_h(x_*,x_0)+ \sigma^2\sum_{t=0}^{k-1}  \frac{c_t}{L_t}. 
\end{eqnarray*}
Dividing by $C_k$ we get the desired result. 
\end{proof}

Theorem~\ref{t: SGD_Ctheorem} itself does not provide insight about the convergence rate of Algorithm~\algS\,, as it strongly depends on the choice of stepsize parameters $\{L_t\}$. We study a suitable choice of stepsize rule in the next subsection.

\subsection{Choice of stepsizes for Theorem~\ref{t: SGD_Ctheorem}}

The goal of this section is to study a choice of stepsize parameters in Theorem~\ref{t: SGD_Ctheorem}. We will analyze separately two cases -- $\mu=0$ and $\mu>0$.

Firstly we start with non-strongly convex case $\mu=0$. The following lemma provides us with the choice of stepsizes minimizing right hand side of \eqref{sgd_dec_Cbound} - stepsizes giving us the best possible convergence rate for Theorem~\ref{t: SGD_Ctheorem}.

\begin{corollary}[Nonstrongly convex rate for Algorithm~\algS]
Suppose that $\mu=0$, i.e. $f$ is convex but not necessarily relative strongly convex. Suppose that we intend to run $k$ iterations of Algorithm \algS\,. Then, constant stepsize controlling parameters $L_t$ given by
\[
L_t=
\frac{\sigma^2L(k-1) } {-\sigma^2+\sqrt{\sigma^4+\sigma^2AL(k-1)}}=\cO(k^{-\frac12})
\]
minimize LHS of \eqref{sgd_dec_Cbound}, obtaining
\begin{equation*}
\sum_{t=1}^k \frac{ \Exp{f(x_t)-f(x_*)}}{k}
\oo{\leq} 
\cO(k^{-\frac12}).
\end{equation*}

\label{cor: sgd O1/sqrt(k)}

\end{corollary}

Note that Stochastic Gradient Descent in the standard smooth setting given by \eqref{eq:sgd_standard} with constant stepsize rule depending on the number of iterations enjoys $O(1/\sqrt{k})$ rate as well \cite{shai_book}. 

Let us now proceed with the case $\mu>0$. Note that the average of iterates of Stochastic Gradient Descent in the standard smooth setting given by \eqref{eq:sgd_standard} with stepsize $\gamma_t=\tfrac{1}{\mu t}$ enjoys $O(\log(k)/k)$ rate \cite{shai_book}. Employing tail averaging technique one can obtain $O(1/k)$ rate \cite{rakhlin2011sgd_tail}.

\begin{lemma}[Choice of stepsizes for Algorithm~\algS]
Suppose that sequence $\{ L_t \}$ is nondecreasing and that sequence $\{c_t\}$ is monotonic for $t\geq T$. In order to attain $O(1/\epsilon)$ rate for stochastic gradient descent we must have $L_t=\Theta(t)$.

\label{L: sgd good choice of stepsizes}
\end{lemma}

Lemma~\ref{L: sgd good choice of stepsizes} provides us with an insight on how stepsizes in Theorem~\ref{t: SGD_Ctheorem} should be chosen in order to attain $O(1/k)$ convergence rate - sequence of stepsize controlling parameters $\{L_t\}$ should be upper and lower bounded by linear function in $t$. A faster or slower rate of increase of $\{L_t\}$ would not result in $O(1/k)$ convergence rate as $k\to \infty$.

The following lemma provides us a bound on convergence rate of Randomized Stochastic Gradient Descent, when sequence $\{L_t\}$ increases linearly with $L_0=L$ i.e. $L_t=L+\alpha t$ for some $\alpha>0$.

\begin{lemma}[Linearly increasing stepsize parameters for Algorithm~\algS]
Consider the convergence rate given by Theorem~\ref{t: SGD_Ctheorem} and stepsize parameters given by $L_t=L+\alpha t$ for some $\alpha>0$. Define 
\begin{equation}
m_\mu\oo{\eqdef}\max\left (\alpha,\mu-\alpha\right ).
\label{eq:m_mu_def}
\end{equation} 
If we choose $\alpha>\mu$ then

\begin{eqnarray*}
C_k 
&\geq & 
   \left (L-\mu\right )^{1-\frac{\mu}{\alpha}}
\frac{\left (L-\mu+\left (k+1\right )\alpha\right )^{\frac{\mu}{\alpha}}-\left (L-\mu+\alpha\right )^{\frac{\mu}{\alpha}}}{\mu},
\\
\sum_{t=0}^{k-1}\frac{c_t}{L_t} 
&\leq &
\frac1L+
 \left (L-\mu+\alpha\right )^{1-\frac{\mu}{\alpha}}
\frac{\left (L-\mu\right )^{\frac{\mu}{\alpha}-1}-\left (L-\mu+k\alpha\right )^{\frac{\mu}{\alpha}-1}}{\alpha-\mu} , 
\end{eqnarray*}
if $\alpha =\mu$ then

\begin{eqnarray*}
C_k 
&= & k ,
\\
\sum_{t=0}^{k-1}\frac{c_t}{L_t}
&\leq &
 \frac{\log\left (L+k\mu\right )-\log\left (L\right )}{\mu} + \frac{1}{L},
\end{eqnarray*}
and finally if $\alpha<\mu$, then
 
\begin{eqnarray*}
C_k 
&\geq &
1+ \frac{\Gamma_{\alpha}\left (L-\mu+\alpha\right )}{\Gamma_{\alpha}\left (L\right )}
\frac{\left (L-m_{\mu}+(k-1)\alpha\right )^{\frac{\mu}{\alpha}}-\left (L-m_{\mu}\right )^{\frac{\mu}{\alpha}}}{\mu},
\\
\sum_{t=0}^{k-1}\frac{c_t}{L_t}
&\leq &
\frac1L+
\frac{\Gamma_{\alpha}\left (L-\mu+\alpha\right )}{\Gamma_{\alpha}\left (L\right )}
   \frac{\left (L+k\alpha\right )^{\frac{\mu}{\alpha}-1}-L^{\frac{\mu}{\alpha}-1}}{\mu-\alpha},
\end{eqnarray*}
for function $\Gamma_{\alpha}$ defined by \eqref{eq:Gamma_def}. 
 In the special case where $\alpha=\tfrac{\mu}{2}$ we obtain
 
 \begin{equation*}
 \sum_{t=1}^k \frac{c_{t-1}}{C_k} \Exp{f\left (x_t\right )-f\left (x_*\right )}
 \oo{\leq} 
\frac{(L-\mu)(L-\frac{\mu}{2})\mu D_h(x_*,x_0)+\sigma^2\mu (1-\frac{\mu}{2L }+k) }{(L+(k-2)\frac{\mu}{2})^2-(L-\frac{\mu}{2})^2+(L-\frac{\mu}{2})\mu},
 \end{equation*}
 where $$
c_t\oo{=}\frac{L+\frac{\mu}{2}\left (t-1\right )}{L-\frac{\mu}{2}},\qquad  C_k\oo{=}\sum_{t=0}^{k-1}c_t.
$$

\label{l:sgd_dec_no_mu}
\end{lemma}

Lemma~\ref{l:sgd_dec_no_mu} provides us with an useful insight on the linearly increasing choice of stepsize controlling parameters in Theorem~\ref{t: SGD_Ctheorem}. We consider the following 3 cases:

\begin{itemize}
\item [$\alpha>\mu$:] since $C_k=\Omega(k^{\mu/\alpha})$ and $\sum_{t=0}^{k-1}\tfrac{c_t}{L_t}=O(1)$, the convergence rate of the weighted sum of errors in functional values is $O(1/k^{\mu/\alpha})$. This is worse than the rate of stochastic gradient descent in standard smooth setting. However, weights from left hand side of the Theorem~\ref{t: SGD_Ctheorem} are decreasing in this case.  

\item [$\alpha=\mu$:] Since $C_k=\Omega(k)$ and $\sum_{t=0}^{k-1}\tfrac{c_t}{L_t}=O(\log(k))$, the convergence rate of the weighted sum of errors in functional values is $O(\log(k)/k)$. Note that weighted sum of errors in objective from left hand side of \eqref{sgd_dec_Cbound} is an average in this case. The average of iterates of Stochastic Gradient Descent with stepsize parameters $L_t=\mu t$ under standard strong convexity assumption enjoys $O(\log(k)/k)$ rate as well \cite{shai_book}.

\item [$\alpha<\mu$:] Since $C_k=\Omega(k^{\mu/\alpha})$ and $\sum_{t=0}^{k-1}\tfrac{c_t}{L_t}=O(k^{\mu/\alpha-1})$  the convergence rate of the weighted sum of errors in functional values is $O(1/k)$. This is as good as the performance of Stochastic Gradient Descent in the standard smooth setting with tail averaging technique \cite{rakhlin2011sgd_tail}. Note that weights from left hand side of the Theorem~\ref{t: SGD_Ctheorem} are increasing, thus we put more value to latter iterates which has a similar effect to the convergence rate as tail averaging in the standard smooth setting. Recall that we use stepsize parameters given by $L_t=L+\alpha t$ for $\alpha<\mu$ in contrast of $L_t=\mu t$ used in \cite{rakhlin2011sgd_tail} (rewritten to our notation).
\end{itemize}

The desired $O(1/k)$ convergence rate is obtained for $\alpha<\mu$. In practice, the condition $\alpha<\mu$ is not trivial to be satisfied, as the relative strong convexity parameter $\mu$ might be unknown and eventually very small. However, this issue can be overcame when strongly convex regularization is used - as we are aware of strongly convex parameter in this case.

\subsection{Minibatch relSGD}
As mentioned previously, if $h$ is $\mu_h$--strongly convex, $\sigma^2$ from the Assumption~\ref{a:sgd_variance} can be chosen so that $\mu_h \sigma^2$ is a global upper bound on the variance of $\widetilde{g}_t$.

Suppose that for $i=1,2,\dots, \mb$ random variables $\widetilde{g}_t^i$ are independent unbiased estimators of $\nabla f(x_t)$ coming from the same distribution. 

Clearly, $\tfrac{1}{\mb}\sum_{i=1}^\mb \widetilde{g}_t^i$ is an unbiased estimator of $\nabla f(x_t)$, thus we can set it in the update rule in Algorithm~\algS\,. Note that $\tfrac{1}{\mb}\sum_{i=1}^\mb \widetilde{g}_t^i$ has $\mb$ times smaller variance comparing to $\widetilde{g}_t^i$ for all $i\leq \mb$. Thus, if we choose $\sigma^2$ such that $\sigma^2 \mu_h$ is an upper bound on the variance, we can allow it to be $\mb$ times smaller when using minibatch of size $\mb$.

\begin{corollary}[Convergence of Minibatch relSGD]
Suppose that $f$ is $L$ smooth and $\mu$ strongly convex relative to $\mu_h$ strongly convex function $h$. Define $c_1=1$ and $c_t=\frac{L_{t-1}}{L_{t}-\mu}c_{t-1}$ for $t\geq 2$ and $C_k=\sum_{t=1}^k c_{t-1}$.
Assume that variance unbiased gradient estimator $ \widetilde{g}_t^i$ of $\nabla f(x_t)$ is upper bounded by $\sigma^2\mu_h$ for all $i\leq \mb$ and $t\leq k$ and also that $\widetilde{g}_t^i$ are independent and identically distributed random variables. Then, iterates of Algorithm~\algS\ with gradient estimator $\tfrac{1}{\mb}\sum_{i=1}^\mb \widetilde{g}_t^i$ satisfy:
\begin{equation*}
\sum_{t=1}^k \frac{c_{t-1}}{C_k} \Exp{f(x_t)-f(x_*)}\oo{\leq} 
\frac{(L_0-\mu)D_h(x_*,x_0)}{C_k}+ \frac{\sigma^2}{\mb}\sum_{t=1}^{k-1}  \frac{c_t}{C_kL_t}.
\end{equation*}
\label{cor: sgd_mb_Cbound}

\end{corollary}

Let us consider a stepsize rule which yields $O(1/k)$ convergence rate as obtained from Lemma \ref{l:sgd_dec_no_mu}. 
In this case, $\mb$--minibatching does not bring linear speedup in terms of the total number of iteration to attain desired accuracy, and thus in terms of the actual work done by the algorithm, it is the best to choose smallest possible minibatch $\mb=1$. 
However, minibatching can be particularly useful in the parallel setup - when one can obtain the a multiple gradient estimator by different processors at the same time.

\section{Experiments}
In this short section we numerically test the convergence of relGD, relRCD and relSGD on two artificial examples, in order to illustrate their potential.

\subsection{An experiment with relRCD}
In this example we compare standard gradient descent to relGD an relRCD. Recall that relRCD is always at most as fast relCD, once it can be applied. 
Our first experiment illustrates the need of relative smoothness assumption - as gradient descent with fixed stepsize applied on the considered function is extremely slow. 

Let us consider a function 
\[
f(x)\oo{\eqdef}\frac{1}{2} x^{\top}Mx+\frac{1}{10}\sum_{i=1}^{100}  \left(x^{(i)}\right)^4,
\]
where $x\in \R^{100}$ and 
\[
M=\frac{A^{\top}A}{\lambda_{\max}(A^{\top}A)}.
\]
Above, $A\in \R^{n\times n}$ is a random matrix with entries from normal distribution with zero mean and variance 1. 

We will use the following reference function 
\[
h(x)\oo{\eqdef} \frac12 \| x\|^2+\frac{1}{10}\sum_{i=1}^{100} \left(x^{(i)}\right)^4.
\]

From Lemma \ref{l:eso_example} we know that $f$ is 1--smooth relative to $h$. On top of that, $(f,\hat{S})\sim \mathrm{ESO}_h(v)$ with $v$ such that $v^{(i)}=\max\left(\tfrac{1}{10},(A^{\top}A)_{ii}\right)$ and uniform sampling $\hat{S}$ such that $\Prob(i\in \hat{S})=1/100$ for all $i$. 
 
In order to compare relGD and relRCD to gradient descent, we need to find a (standard) smoothness parameter $L$. For this purpose, we will restrict the domain as $\{x \,|\, \|x\|_{\infty}^{2}\leq  2\|x_0\|_{\infty}^{2}\}$. Clearly, $\frac{1}{2} x^{\top}Mx$ is 1--smooth and maximal eigenvalue of hessian of $\tfrac{1}{10}\sum_{i=1}^{100}  \left(x^{(i)}\right)^4$ is $\tfrac{12}{10}\|x\|_{\infty}^{2}$. We set $x_0$ to be random vector with independent zero mean entries with variance $10^{6}$. Thus, $L$ is in the order of $10^{6}$ in contrast to relative smoothness parameter, which is 1. 
The plot below illustrates a convergence result of gradient descent, relGD and relRCD for the artificial setting that we just described. 
 
\begin{figure}[H]
  \centering
  \includegraphics[width=0.6\linewidth]{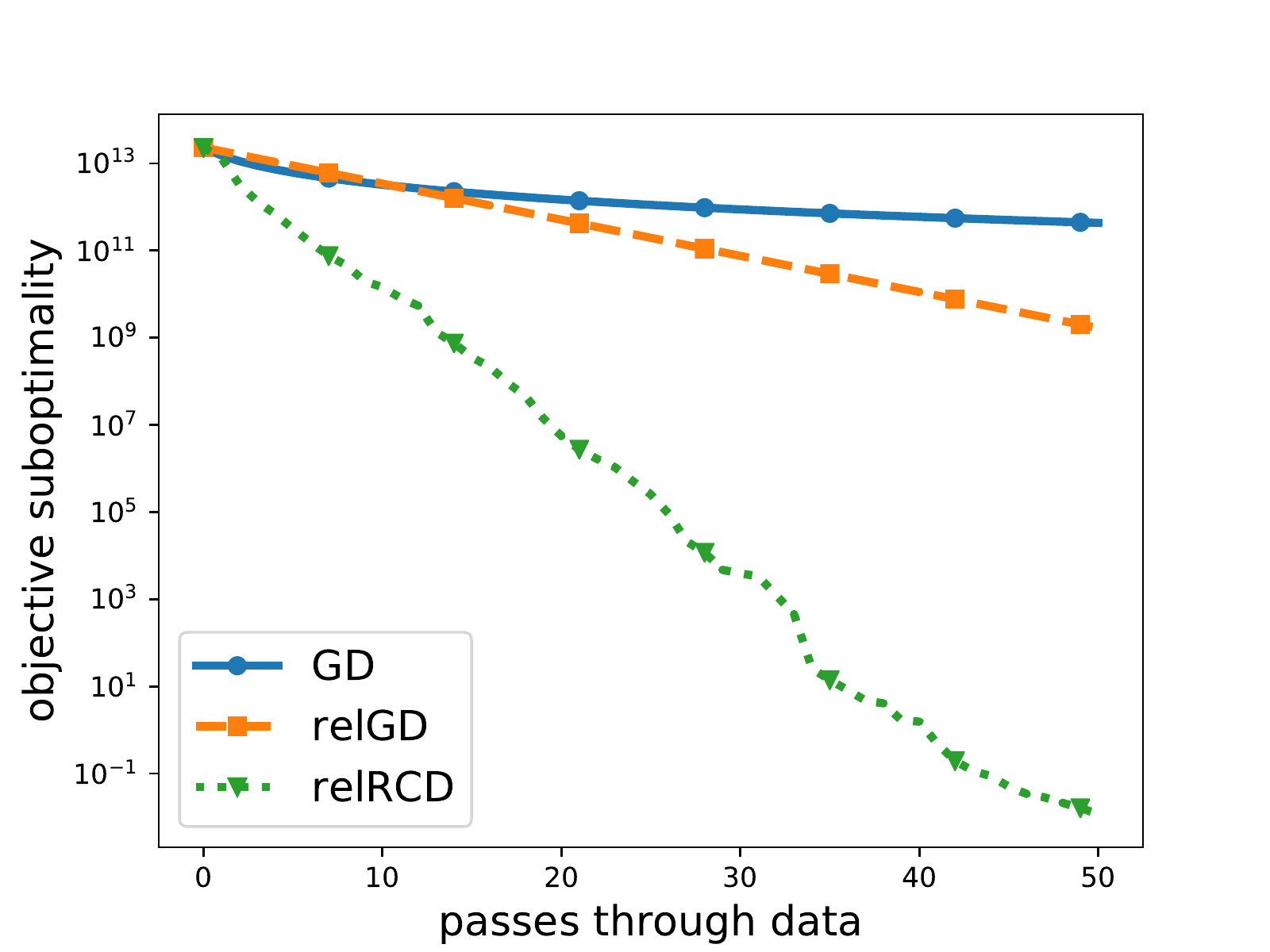}
  \caption{Comparison of Gradient descent to relGD and relRCD}
  \label{fig:rcd}
  \end{figure}

Figure~\ref{fig:rcd} show that the algorithms behave as we expected from the theory - Gradient descent has a faster start first few epochs, which is due to the fact that smoothness parameter $L$ is huge but still tight in the region far from the optimum. However, with increasing number of iterations, Gradient descent is significantly outperformed by the other two algorithms. Notice that relRCD enjoys here the best convergence rate, which is expected from the theory since ESO parameters $v$ are smaller than relative smoothness parameter. In this specific case maximal element of $v$ is 0.36.  

\subsection{An experiment with relSGD}

In this experiment we compare relGD to relSGD for various choice of stepsize parameters $L_t$.

Let us consider Poisson linear inverse problem, where one minimizes Kullback-Liebler divergence between $b$ and $Ax$:
\begin{eqnarray*}
\min_x  && f(x) \eqdef \sum_{i=1}^m f^{(i)}(x)\eqdef\sum_{i=1}^m\Big( b^{(i)} \log \frac{b^{(i)}}{(Ax)^{(i)}}+(Ax)^{(i)}-b^{(i)} \Big)
\\
\text{subject to} && 0 < x^{i}, \, \forall \,i,
\end{eqnarray*}
where $b\in \R^m_{++}$ and matrix $A\in \R^{m\times n}_+$ have nonzero rows. 
In \cite{bauschke2016NoLips}, it was shown that $f$ is $L\eqdef\sum_{i=1}^m b^{(i)}$--smooth with respect to Burg's entropy $h(x)\eqdef-\sum_{i=1}^m \log(x^{(i)})$. 

We consider here $m\nabla f^{(i)}(x)$ for randomly chosen $i$ to be an unbiased gradient estimator. Notice that the access to stochastic oracle is it is $m$ times cheaper comparing to the cost of the full gradient due to ERM structure. 

\begin{figure}[H]
  \centering
  \includegraphics[width=0.6\linewidth]{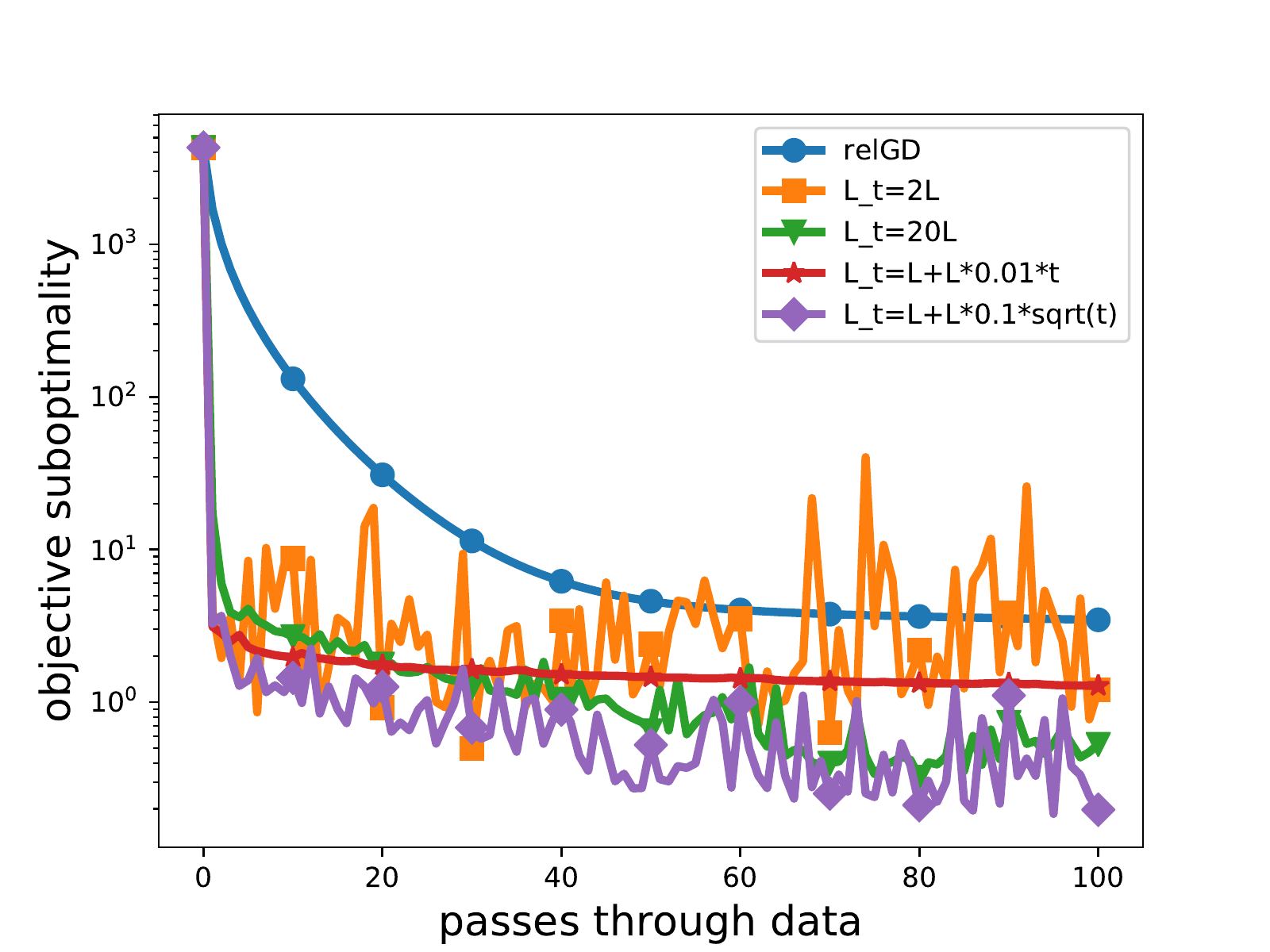}
  \caption{Comparison of relGD and relSGD for $A=|A'|,\,b=|b'|,\, x_0=|x_0'|$ where $A',\,b',\, x_0'$ are vectors (or matrix) with entries randomly generated from normal distribution with zero mean and variance 1.  }
  \label{fig:sgd}
  \end{figure}

Figure~\ref{fig:sgd} illustrates $O(1/k)$ convergence rate (sublinear) of relGD. We can clearly see that relSGD performs much faster first few passes through data, however for smaller constant $L_t$ it oscillates and seems not converge to the optimum, as expected from theory. Larger constant $L_t$ yields a very slightly slower decrease at the beginning but it is as expected less noisy and give us a better approximation after more passes through data. On the other hand, linearly increasing parameters $L_t$ seems to be too fast, as the convergence significantly slows with the increased number of iterations. We obtained the best behaviour for $L_t=\tfrac{L}{10}\sqrt{t}$, which is expected since Corollary \ref{cor: sgd O1/sqrt(k)} claims that optimal stepsize controlling parameters are $O(1/\sqrt{k})$ for non-strongly convex case, i.e. $\mu=0$.

\section{Conclusions and Extensions}

In this work, we presented first stochastic primal algorithms for minimizing Relatively smooth functions. We bridge the well developed area of stochastic smooth optimization with fresh area of relative smooth optimization. This way, we also contribute to better understanding of mirror descent, obtaining the first stochastic mirror descent type algorithm with linear convergence rate. 
However, there is still a plenty of space to extend on the results of our work. We give here few examples. 

\begin{itemize}
\item Arbitrary Sampling for relRCD. In this work we showed the convergence of Randomized Coordinate Descent under ESO assumption for uniform sampling strategies. However, Randomized Coordinate Descent under standard smoothness allows arbitrary sampling strategy \cite{NSync}, which can potentially be extended to relative smooth setting as well, and therefore to gain additional speedup from importance sampling. 

\item Variance reduced relSGD for empirical risk minimization. RelSGD converges since the sequence of stepsize controlling parameters $\{L_t\}$ goes to infinity. However, for Empirical Risk Minimization problem in standard smooth setting, one can attain a linear convergence using variance reduction techniques \cite{sag, SAGA, svrg, SDCA, sarah}, as we mentioned earlier. 

\item Application. In this work we provide only theoretical results on the algorithm and convergence rates. We did not give any application of our algorithms to a particular problem, however we believe that this work might help to solve a various optimization challenges in practice, especially since it brings a different insights on under which conditions can stochastic mirror descent perform extremely fast.

\end{itemize}

\bibliographystyle{plain} 
\bibliography{literature}
\newpage

\appendix

\section{Key technical lemmas}

For completeness, we firstly give proof of Three point property. 

\subsection{Proof of the three point property}

Note that $\phi(x)+D_h(x,z)$ is differentiable and convex in $x$. Using the definition of $z_+$ we have
$$\langle \nabla\phi(z_+)+\nabla h(z_+)-\nabla h(z),x-z_+  \rangle \oo{\geq 0},\ \forall x\in Q . $$
Using definition of $D_h(\cdot,\cdot)$ we can see that
$$\langle \nabla h(z_+)-\nabla h(z),x-z_+  \rangle\oo{=} D_h(x,z)-D_h(z_+,z)-D_h(x,z_+).$$
Putting the above together, we see that

\begin{eqnarray*}
0
&\leq&
 \langle \nabla\phi(z_+)+\nabla h(z_+)-\nabla h(z),x-z_+  \rangle \\
& =& 
 D_h(x,z)-D_h(z_+,z)-D_h(x,z_+) +\langle \nabla\phi(z_+),x-z_+  \rangle \\
 &\leq &
  D_h(x,z)-D_h(z_+,z)-D_h(x,z_+) + \phi(x)-\phi(z_+).
\end{eqnarray*}
The last inequality is due to convexity of $\phi$.

\subsection{Key lemma for analysis}

The following lemma allow us to get a convergence rate for Algorithms 

\begin{lemma}
Suppose that for positive sequences $\{f_t\}, \{D_t\} $ we have 

\begin{equation}
f_{t+1}\oo{\leq} (1-\delta)f_t+ \delta  f_*+\left(\varphi-\delta\psi\right)D_t-\varphi D_{t+1},
\label{eq:iter_general_tech}
\end{equation}
where $\delta,\,\varphi,\,\psi \in \R$ satisfy $1\geq \delta>0$ and $\varphi\geq \psi>0$. Then, the following inequality holds

\begin{equation*}
\sum_{t=1}^{k}c_t \big(f_t-f_*\big)
\leq
    \frac{(\varphi-\delta\psi)D_0 +(1-\delta)(f_0-f_*)}{1-\frac{\varphi}{\psi}+\frac{\varphi}{\psi}\Big ( \frac{\varphi}{\varphi-\delta\psi} \Big )^{k-1}}, 
\end{equation*}
where $c_t\eqdef C_t/\sum_{t=1}^kC_t$ for

\[
C_t
\oo{\eqdef}
\begin{cases}\left( \frac{\varphi}{\varphi-\delta\psi} \right )^{t-1} \frac{\varphi-\psi}{\delta^{-1}\varphi-\psi}, & 1\leq t\leq k-1 \\
\left ( \frac{\varphi}{\varphi-\delta\psi} \right )^{k-1}, & t=k.
\end{cases}
\]

\label{l:rate_from_iter}
\end{lemma}

\begin{proof}
Let us multiple the inequality \eqref{eq:iter_general_tech} by $\big(\frac{\varphi}{\varphi-\delta\psi}\big)^t$ for iterates $t=0,1,\dots,k-1$ and sum them:

\begin{eqnarray*}
\sum_{t=0}^{k-1} \left ( \frac{\varphi}{\varphi-\delta\psi} \right)^t f_{t+1} 
&\leq& 
\sum_{t=0}^{k-1} \left ( \frac{\varphi}{\varphi-\delta\psi} \right )^t \Big (   (10\delta){n}f_t+ \delta  f_*\Big )
\\
&& \qquad
+\sum_{t=0}^{k-1} \left ( \frac{\varphi}{\varphi-\delta\psi} \right )^t  \left ( \left(\varphi-\delta\psi\right)D_{t}-\varphi D_{t+1}
\right).
\end{eqnarray*}

Rearranging the terms, we get
\begin{eqnarray}
&&\sum_{t=0}^{k-1}  \Big ( \frac{\varphi}{\varphi-\delta\psi} \Big )^t\Big ( f_{t+1}- (1-\delta)f_t- \delta  f_* \Big) 
\label{eq:rcd_mb_partial_0}
\\
&&\qquad \qquad \leq \quad
 \left(\varphi-\delta\psi\right)D_0-\Big ( \frac{\varphi}{\varphi-\delta\psi} \Big )^{k-1}\varphi D_k
 \nonumber
 \\
 &&\qquad \qquad  \leq \quad
  \left(\varphi-\delta\psi\right)D_0.
 \label{eq:rcd_mb_th_summed}
\end{eqnarray}

For simplicity, throughout this proof denote $r_t=f_{t}-f_*$. Let us continue with the bound above:

\begin{eqnarray}
 \left(\varphi-\delta\psi\right)D_0
&\stackrel{\eqref{eq:rcd_mb_th_summed}}{\geq} & 
\left( \frac{\varphi}{\varphi-\delta\psi} \right )^{k-1} f_k+
\sum_{t=1}^{k-1} \left( \frac{\varphi}{\varphi-\delta\psi} \right )^{t-1} \left( f_{t}- (1-\delta)\frac{\varphi}{\varphi-\delta\psi}f_t\right)
\nonumber
\\
&&
\nonumber
\qquad
-(1-\delta)f_0 
-\delta\sum_{t=0}^{k-1}  \left ( \frac{\varphi}{\varphi-\delta\psi} \right )^t f_* 
\\
\nonumber
&\stackrel{}{=}& 
\left ( \frac{\varphi}{\varphi-\delta\psi} \right )^{k-1} f_k+
\sum_{t=1}^{k-1} \left ( \frac{\varphi}{\varphi-\delta\psi} \right )^{t-1} \frac{\varphi-\psi}{\delta^{-1}\varphi-\psi}f_t
\\
&&
\qquad
-(1-\delta)f_0 
-\delta\sum_{t=0}^{k-1}  \left ( \frac{\varphi}{\varphi-\delta\psi} \right )^t f_* 
\label{eq:rcd_mb_partial_1}
\\
&\stackrel{(*)}{=}&  
\left ( \frac{\varphi}{\varphi-\delta\psi} \right )^{k-1} r_{k}+
\sum_{t=1}^{k-1} \left ( \frac{\varphi}{\varphi-\delta\psi} \right )^{t-1} \frac{\varphi-\psi}{\delta^{-1}\varphi-\psi}r_t -(1-\delta)r_0.
\label{RCD_tel}
\end{eqnarray}

Equality $(*)$ is obtained by the fact that the sum of terms corresponding to $f(\cdot)$ is 0 (this can be easily seen as it is equal to \eqref{eq:rcd_mb_partial_0}).

Recall that we have 
\[
C_t
\oo{=}
\begin{cases}\left( \frac{\varphi}{\varphi-\delta\psi} \right )^{t-1} \frac{\varphi-\psi}{\frac{n}{\mb}\varphi-\psi}, & 1\leq t\leq k-1 \\
\left ( \frac{\varphi}{\varphi-\delta\psi} \right )^{k-1}, & t=k.
\end{cases}
\]

and $c_t\eqdef C_t/\sum_{t=1}^kC_t$. Since the sum of terms corresponding to $f_t$ for some $t$ or $f_*$ in \eqref{eq:rcd_mb_partial_1} is 0 (because it is equal to \eqref{eq:rcd_mb_partial_0}), we have

\begin{eqnarray}
\nonumber
\sum_{t=1}^k C_t
&=&
\left ( \frac{\varphi}{\varphi-\delta\psi} \right )^{k-1} +
\sum_{t=1}^{k-1} \left ( \frac{\varphi}{\varphi-\delta\psi} \right )^{t-1} \frac{\varphi-\psi}{\frac{n}{\mb}\varphi-\psi}
\\
\nonumber
&=&
 (1-\delta) +\delta\sum_{t=0}^{k-1}  \Big ( \frac{\varphi}{\varphi-\delta\psi} \Big )^t.
 \\
 \nonumber
 &=&
  (1-\delta) +
\delta\frac{\Big ( \frac{\varphi}{\varphi-\delta\psi} \Big )^{k}-1}{\frac{\varphi}{\varphi-\delta\psi}-1} 
\\
\nonumber
&=&
 (1-\delta) +
\frac{\Big ( \frac{\varphi}{\varphi-\delta\psi} \Big )^{k}-1}
{\frac{\psi}{\varphi-\delta\psi}} 
\\
\nonumber
&=&
 (1-\delta) +
\left(\varphi-\delta\psi\right)\frac{\Big ( \frac{\varphi}{\varphi-\delta\psi} \Big )^{k}-1}
{\psi} 
\\
&=&
1-\frac{\varphi}{\psi}+\frac{\varphi}{\psi}\left ( \frac{\varphi}{\varphi-\delta\psi} \right )^{k-1}.
\label{sum_weights}
\end{eqnarray}

Thus, we can rewrite \eqref{RCD_tel} as follows
\begin{eqnarray*}
\sum_{t=1}^{k}c_tr_t 
&\stackrel{\eqref{RCD_tel}}{\leq} & 
\bigg (\left(\varphi-\delta\psi\right)D_0 +(1-\delta)r_0\bigg)\frac{1}{\sum_{t=1}^kC_t}
\\
&\stackrel{\eqref{sum_weights}}{=}&
  \bigg (\left(\varphi-\delta\psi\right)D_0 +(1-\delta)r_0\bigg) \frac{1}{1-\frac{\varphi}{\psi}+\frac{\varphi}{\psi}\Big ( \frac{\varphi}{\varphi-\delta\psi} \Big )^{k-1}}.
\end{eqnarray*}

\end{proof}

\section{Proofs for Section~\ref{S:sgd}}

\subsection{Proof of Corollary~\ref{cor: sgd O1/sqrt(k)}}

Denote $l_t=(L_t)^{-1}$ for simplicity. It is easy to see that 
$$c_t=L\,l_t,\quad C_k=1+L\sum_{t=1}^{k-1}l_t,\quad \sum_{t=0}^{k-1} c_tl_t=L+L\left(\sum_{t=1}^{k-1} l_t^2\right).$$
Denote 
\[
A=(L-\mu)D_h(x_*,x_0)+\sigma^2L.
\]
Minimizing RHS of \eqref{sgd_dec_Cbound} to obtain the best rate is equivalent to minimize 
\[
\frac{A+\sigma^2L\left(\sum_{t=1}^{k-1} l_t^2\right)}{1+L\sum_{t=1}^{k-1}l_t}. 
\]
Notice that the expression above is minimized for constant $l_t$, as if $l_t\neq l_s$, setting $l_t=l_s=\frac{l_t+l_s}{2}$ leads to strictly smaller value of the expression. Therefore, it suffices to minimize
\[
\frac{A+\sigma^2L(k-1) l^2}{1+L(k-1) l}
\]
in $l$. First order optimality condition yields
\[
2\sigma^2 L(k-1)l (1+L(k-1)l)=(A+\sigma^2 L(k-1)l^2)L(k-1),
\]
which is equivalent to
\[
\sigma^2L(k-1) l^2 +2\sigma^2l-A=0.
\]
The quadratic equation above have a single solution
\[
l=\frac{-\sigma^2+\sqrt{\sigma^4+\sigma^2AL(k-1)}}{\sigma^2L(k-1) } ,
\]
which finishes the proof.

\subsection{Proof of Lemma~\ref{L: sgd good choice of stepsizes}}

For simplicity, denote $l_t=(L_t^{-1})$. Thus, $\{l_t\}$ is nonincreasing sequence. 
Note that the rate from the Theorem~\ref{t: SGD_Ctheorem} is $O(1/k)$ if and only if both
$$ \frac{1}{C_k} \quad \text{and} \quad \sum_{t=0}^{k-1} \frac{c_tl_t}{C_k}$$
are $O(1/k)$. 

Let us now consider that $\{c_t\}$ is nonincreasing for $t\geq T$. 
Suppose that 
\[
1>\liminf \frac{c_{t}}{c_{t-1}}\eqdef r_c.
\]
Then for all $k$ there is $K\geq k$ such that \[1>\frac{1+r_c}{2}>\frac{c_{K}}{c_{K-1}}.\] 
Thus there is infinitely many $t$ such that
 \[1>\frac{1+r_c}{2}>\frac{c_{t}}{c_{t-1}}.\] 
Since $\{c_t\}$ is nonincreasing for $t\geq T$, we have that $\{c_t\}\to 0$ which is a contradiction with the assumption that $\tfrac{1}{C_t}=O(1/t)$. Thus we have

\[
1=\liminf \frac{c_{t}}{c_{t-1}}=\lim \frac{c_{t}}{c_{t-1}}, 
\]
which implies that
\[
\lim L_t-L_{t-1}=\mu.
\]
The above means that $L_t=\Theta(t)$.
We have just proven the lemma for asymptotically nonincreasing $\{c_t\}$.

Now, suppose that $\{c_t\}$ is increasing sequence for $t\geq T$. Then we have for all $t\geq T$ 
\[\tfrac{L_{t-1}}{L_t-\mu}>1.\] Thus $L_t<L_{t-1}+\mu$, which implies that $L_t=O(t)$ and $l_t=\Omega(1/t)$.

On the other hand, looking at $\sum_{t=0}^{k-1} \tfrac{c_tl_t}{C_k}$ as the weighted sum of $l_t$, since $l_{k-1}$ is the smallest from $\{l_t \}$ we immediately have
$$ O(1/k)= \sum_{t=0}^{k-1} \frac{c_tl_t}{C_k} \geq l_{k-1}\geq l_k,$$
which means that $l_t=O(1/t)$. Thus, $l_t=\Theta(1/t)$ and $L_t=\Theta(t)$.

\subsection{Proof of Lemma~\ref{l:sgd_dec_no_mu}}
First, we introduce two technical lemmas. 

\begin{lemma}

Let us fix $\alpha>0$. There exist a convex continuous function $\gamma_\alpha(x)$ on $\R_+$ such that for all $x>0$ we have 

\begin{equation}
\gamma_{\alpha}(x+\alpha)=\log(x)+\gamma_{\alpha}(x).
\label{eq:gamma_def}
\end{equation}
\label{l:log_gamma}
\end{lemma}

\begin{proof}
We will construct function $\gamma_{\alpha}$ in the following way - Let us set $\gamma_{\alpha}(x)=0$ for $x\in [1,1+\alpha)$. For $x\geq 1+\alpha$ let us set recursively $\gamma_{\alpha}(x+\alpha)=\log(x)+\gamma_{\alpha}(x)$ and for $ x<1 $ let us set $\gamma_{\alpha}(x)=-\log(x)$. Clearly, equality $\eqref{eq:gamma_def}$ holds.

We will firstly prove that $\gamma_{\alpha}$ is continuous on $\R_+$ and differentiable on $R_+\backslash \{1\}$. Let us start with intervals $[1+k\alpha,1+(k+1)\alpha)$ for all $k$. 

Clearly, $\gamma_{\alpha}$ it is continuous and differentiable on $[1,1+\alpha)$. Suppose now inductively that $\gamma_{\alpha}$ is continuous and differentiable on $[1+k\alpha,1+(k+1)\alpha)$ for some $k\geq 0$. Then, for $x\in [1+(k+1)\alpha,1+(k+2)\alpha)$ we have
\[
\gamma_{\alpha}(x)=\log(x-\alpha)+\gamma_{\alpha}(x-\alpha). 
\]
Since both $\log(x-\alpha)$ and $\gamma_{\alpha}(x-\alpha)$ are continuous and differentaible functions on $[1+(k+1)\alpha,1+(k+2)\alpha)$, $\gamma_{\alpha}(x)$ is also continuous and differentaible on $[1+(k+1)\alpha,1+(k+2)\alpha)$. 

Clearly, $\gamma_{\alpha}$ it is continuous and differentiable on $(0,1)$.

It remains to show continuity and differentiability in the points $\{1+k\alpha\}$ for $k\geq 1$ and continuity in $\{1\}$. It is a simple exercise to see the continuity and differentiability in $\{1+\alpha\}$. For $1+k\alpha$ where $k\geq 2$ we can show it inductively -- 
as $\gamma_{\alpha}(x-\alpha)$ and $\log(x-\alpha)$ are continuous and differentiable on $(1+(k-\tfrac12)\alpha,1+(k+\tfrac12)\alpha)$, 
then $\gamma_{\alpha}(x)$ is continuous and differentiable on $(1+(k-\tfrac12)\alpha,1+(k+\tfrac12)\alpha)$ as well and thus it is continuous and differentiable in point $\{1+k\alpha\}$. 
On top of that, $\gamma_\alpha$ is clearly continuous in $\{1\}$. 

We have just proven that $\gamma_\alpha$ is continuous on $\R_+$ and differentiable on $\R_+\backslash \{1\}$.

Now we can proceed with the proof of convexity. We will show that the (sub)derivative of $\gamma_{\alpha}$ is nonegative for all $x>0$. Clearly, $\gamma_{\alpha}'(x)\geq 0$ for $x\in (0,1)$ and subdifferential in $\{1\}$ is nonegative as well. Let us write $x=1+\{x\}_\alpha+k\alpha$, where $0\leq \{x\}_\alpha<\alpha$ and $k\geq -1$. Then we have 

\begin{eqnarray*}
\gamma_{\alpha}'(x)
&=&
\lim_{\epsilon \to 0} 
\frac{ \gamma_{\alpha}(x+\epsilon)-\gamma_{\alpha}(x)}{\epsilon}
\\
&=&
\lim_{\epsilon \to 0} 
\frac{\sum_{i=0}^{k-1}\left(
\log\left( 1+\{x\}_\alpha+i\alpha+\epsilon\right)-
\log\left(1+\{x\}_\alpha+i\alpha\right)
\right)
}
{\epsilon}
\\
&& \qquad
+\frac{\gamma_{\alpha}(1+\{x\}_\alpha+\epsilon)-\gamma_{\alpha}(1+\{x\}_\alpha)}
{\epsilon}
\\
&\stackrel{(*)}{=}&
\lim_{\epsilon \to 0} 
\frac{\sum_{i=0}^{k-1}\left(
\log\left( 1+\{x\}_\alpha+i\alpha+\epsilon\right)-
\log\left(1+\{x\}_\alpha+i\alpha\right)
\right)}
{\epsilon}
\\
&\stackrel{(**)}{\geq}&
0.
\end{eqnarray*}

Equality $(*)$ holds since for small enough $\epsilon$ we have $1+\{x\}_\alpha+\epsilon <2\alpha$ and inequality $(**)$ holds due to the fact that logarithm is an increasing function.

\end{proof}

Denote 
\begin{equation}
\Gamma_{\alpha}(x)\eqdef\exp(\gamma_{\alpha}(x))
\label{eq:Gamma_def}
\end{equation}
 for $\gamma_{\alpha}$ given from Lemma \ref{l:log_gamma}. Thus, $\Gamma_{\alpha}$ is log-convex function satisfying 

\begin{equation}
\Gamma_{\alpha}(x+\alpha)=x\Gamma_{\alpha}(x).
\label{eq:gamma_alpha_rec}
\end{equation}

Note that when $\alpha=1$, function $\gamma$ can be chosen as log Gamma function and thus $\Gamma_1$ can be chosen to be standard Gamma function.

The following lemma is crucial for our analysis, allowing us to bound the ratio of functions $\Gamma_{\alpha}(\cdot)$ with nearby arguments.

\begin{lemma}
Consider a function $\Gamma_{\alpha}$ defined above. Then, we have for all $0\leq s\leq\alpha$ and $x>0$:

\begin{equation}
x^{1-\frac{s}{\alpha}}
\leq
\frac{\Gamma_{\alpha}(x+\alpha)}{\Gamma_{\alpha}(x+s)}
\leq
(x+\alpha)^{1-\frac{s}{\alpha}}.
\label{eq:gautschi_alpha}
\end{equation}
\label{l:gautschi_alpha}
\end{lemma}

\begin{proof}

Using convexity of $\gamma_{\alpha}$ we have
\[
\Gamma_{\alpha}(x+s)
\leq
 \Gamma_{\alpha}(x)^{1-\frac{s}{\alpha}} \Gamma_{\alpha}(x+\alpha)^{\frac{s}{\alpha}}
 \stackrel{\eqref{eq:gamma_alpha_rec}}{=}
 x^{\frac{s}{\alpha}-1}\Gamma_{\alpha}(x+\alpha).
\]
Rearranging the above we obtain
\[
x^{1-\frac{s}{\alpha}} 
\leq
\frac{\Gamma_{\alpha}(x+\alpha)}
{\Gamma_{\alpha}(x+s)} 
.\]

On the other hand, using convexity of $\gamma_{\alpha}$ again we obtain

\[
\Gamma_{\alpha}(x+\alpha)
\leq
\Gamma_{\alpha}(x+s)^{\frac{s}{\alpha}}
\Gamma_{\alpha}(x+s+\alpha)^{1-\frac{s}{\alpha}}
\stackrel{\ref{eq:gamma_alpha_rec}}{=}
(x+s)^{1-\frac{s}{\alpha}}\Gamma_{\alpha}(x+s).
\]
By rearranging the above, we get
\[
\frac{\Gamma_{\alpha}(x+\alpha)}
{\Gamma_{\alpha}(x+s)} 
\leq
(x+s)^{1-\frac{s}{\alpha}}
\leq
(x+\alpha)^{1-\frac{s}{\alpha}}
.
\]
\end{proof}

We can now proceed with the proof of Lemma~\ref{l:sgd_dec_no_mu} itself. 

\begin{proof}
Note that
\begin{eqnarray}
c_t
\nonumber
&=&
\prod_{i=0}^{t-1}\frac{L_i}{L_{i+1}-\mu}
\\
\nonumber
&\stackrel{\eqref{eq:gamma_alpha_rec}}{=}&
\frac{\frac{\Gamma_{\alpha}(L+t\alpha)}{\Gamma_{\alpha}(L)}}{\frac{\Gamma_{\alpha}(L+(t+1)\alpha-\mu)}{\Gamma_{\alpha}(L-\mu+\alpha)}}
\\
&=&
\frac{\Gamma_{\alpha}(L-\mu+\alpha)}{\Gamma_{\alpha}(L)}
\frac{\Gamma_{\alpha}(L+t\alpha)}{\Gamma_{\alpha}(L-\mu+(t+1)\alpha)}.
\label{eq:c_eq_using_gamma_alpha}
\end{eqnarray}

Let us firstly consider the case when $\alpha>\mu$. Choosing $x= L-\mu +t\alpha$ and $s=\mu$ in \eqref{eq:gautschi_alpha} we get

$$ 
(L-\mu+t\alpha)^{1-\frac{\mu}{\alpha}}
\leq
\frac{\Gamma_{\alpha}(L-\mu+(t+1)\alpha)}{\Gamma_{\alpha}(L+t\alpha)}
\leq
(L-\mu+(t+1)\alpha)^{1-\frac{\mu}{\alpha}}. 
$$

The inequality above allows us to get the following bound on $c_t$

\begin{equation}
\frac{\Gamma_{\alpha}(L-\mu+\alpha)}{\Gamma_{\alpha}(L)}(L-\mu+t\alpha)^{\frac{\mu}{\alpha}-1}
\geq 
c_t
\geq
\frac{\Gamma_{\alpha}(L-\mu+\alpha)}{\Gamma_{\alpha}(L)}(L-\mu+(t+1)\alpha)^{\frac{\mu}{\alpha}-1}. 
\label{eq:c_bound_alpha}
\end{equation}

 Clearly, $\{c_t\}$ is decreasing and thus using the bound above we obtain

\begin{eqnarray*}
C_k
&= & \sum_{t=0}^{k-1}c_t 
\quad 
\stackrel{\eqref{eq:c_bound_alpha}}{\geq} 
\quad 
\sum_{t=0}^{k-1}
\frac{\Gamma_{\alpha}(L-\mu+\alpha)}{\Gamma_{\alpha}(L)}
(L-\mu+(t+1)\alpha)^{\frac{\mu}{\alpha} -1}
\\
&=&
\frac{\Gamma_{\alpha}(L-\mu+\alpha)}{\Gamma_{\alpha}(L)}
\sum_{t=0}^{k-1}
(L-\mu+(t+1)\alpha)^{\frac{\mu}{\alpha} -1}
\\
&\stackrel{(*)}{\geq} & 
\frac{\Gamma_{\alpha}(L-\mu+\alpha)}{\Gamma_{\alpha}(L)}
\int_{0}^{k} 
(L-\mu+(t+1)\alpha)^{\frac{\mu}{\alpha} -1} 
dt
\\
&= &
\frac{\Gamma_{\alpha}(L-\mu+\alpha)}{\Gamma_{\alpha}(L)}
\int_{0}^{(k)\alpha} 
(L-\mu+\alpha+t)^{\frac{\mu}{\alpha} -1}
\frac{1}{\alpha} 
dt
\\
&=&
\frac{\Gamma_{\alpha}(L-\mu+\alpha)}{\Gamma_{\alpha}(L)}
\frac{1}{\alpha}
\Big[\frac{(L-\mu+\alpha+t)^{\frac{\mu}{\alpha}}}{\frac{\mu}{\alpha}} \Big]_{t=0}^{k\alpha}
\\
&=&
\frac{\Gamma_{\alpha}(L-\mu+\alpha)}{\Gamma_{\alpha}(L)}
\frac{(L-\mu+(k+1)\alpha)^{\frac{\mu}{\alpha}}-(L-\mu+\alpha)^{\frac{\mu}{\alpha}}}{\mu}
\\
&\stackrel{\eqref{eq:gautschi_alpha}}{\geq} &
 (L-\mu)^{1-\frac{\mu}{\alpha}}
\frac{(L-\mu+(k+1)\alpha)^{\frac{\mu}{\alpha}}-(L-\mu+\alpha)^{\frac{\mu}{\alpha}}}{\mu}.
\end{eqnarray*}

Inequality $(*)$ holds since $(L-\mu+(t+1)\alpha)^{\mu/\alpha-1}$ is decreasing in $t$. 
On the other hand, we have 

\begin{eqnarray*} 
\sum_{t=1}^{k-1}\frac{c_t}{L_t}
&\stackrel{\eqref{eq:c_bound_alpha}}{\leq} &
\sum_{t=1}^{k-1} \frac{\Gamma_{\alpha}(L-\mu+\alpha)}{\Gamma_{\alpha}(L)}
(L-\mu+t\alpha)^{\frac{\mu}{\alpha} -1}
 \frac{1}{L+t\alpha}
\\
&=&
\frac{\Gamma_{\alpha}(L-\mu+\alpha)}{\Gamma_{\alpha}(L)}
\sum_{t=1}^{k-1} \frac{1}{L+t\alpha}
(L-\mu+t\alpha)^{\frac{\mu}{\alpha} -1}
\\
&\stackrel{(*)}{\leq} & 
\frac{\Gamma_{\alpha}(L-\mu+\alpha)}{\Gamma_{\alpha}(L)}\sum_{t=1}^{k-1} 
(L-\mu+t\alpha)^{\frac{\mu}{\alpha} -2}
\\
&\stackrel{(**)}{\leq} & 
\frac{\Gamma_{\alpha}(L-\mu+\alpha)}{\Gamma_{\alpha}(L)}\int_0^k 
(L-\mu+t\alpha)^{\frac{\mu}{\alpha} -2}
 dt
\\
&= &
\frac{\Gamma_{\alpha}(L-\mu+\alpha)}{\Gamma_{\alpha}(L)}
\int_0^{k\alpha} 
(L-\mu+t)^{\frac{\mu}{\alpha} -2} \frac{1}{\alpha}
 dt
 \\
 &=&
\frac{\Gamma_{\alpha}(L-\mu+\alpha)}{\Gamma_{\alpha}(L)}
\frac{1}{\alpha} \left[
\frac{(L-\mu+t)^{\frac{\mu}{\alpha}-1}}{\frac{\mu}{\alpha}-1} 
\right]_0^{k\alpha}
 \\
&=&
\frac{\Gamma_{\alpha}(L-\mu+\alpha)}{\Gamma_{\alpha}(L)} \frac{(L-\mu)^{\frac{\mu}{\alpha}-1}-(L-\mu+k\alpha)^{\frac{\mu}{\alpha}-1}}{\alpha-\mu} 
 \\
 &\stackrel{\eqref{eq:gautschi_alpha}}{\leq} &
 (L-\mu+\alpha)^{1-\frac{\mu}{\alpha}}
\frac{(L-\mu)^{\frac{\mu}{\alpha}-1}-(L-\mu+k\alpha)^{\frac{\mu}{\alpha}-1}}{\alpha-\mu}.
\end{eqnarray*}

Inequality $(*)$ holds due to the fact that $(L+t\alpha)^{-1}\leq(L-\mu+t\alpha)^{-1}$ and inequality $(**)$ holds since $(L-\mu+t\alpha)^{\mu/\alpha-2}$ is decreasing in $t$. 
Thus we have 
\[
\sum_{t=0}^{k-1}\frac{c_t}{L_t}
\leq
\frac1L
+
(L-\mu+\alpha)^{1-\frac{\mu}{\alpha}}
\frac{(L-\mu)^{\frac{\mu}{\alpha}-1}-(L-\mu+k\alpha)^{\frac{\mu}{\alpha}-1}}{\alpha-\mu}.
\]
and we have just proven the first part of the lemma.

Let us now look at the case when $\alpha\leq\mu$. It will be useful to denote $\lfloor \mu\rfloor_\alpha$ as the largest integer such that $\mu - \lfloor \mu\rfloor_\alpha \alpha$ is positive. Denote also \[\{\mu\}_\alpha\eqdef \mu - \lfloor \mu\rfloor_\alpha \alpha.\]
Using \eqref{eq:gamma_alpha_rec} we obtain

\begin{eqnarray*}
\frac{\Gamma_{\alpha}(L+t\alpha)}{\Gamma_{\alpha}(L-\mu+(t+1)\alpha)}
&=&
\frac{\Gamma_{\alpha}(L+t\alpha)(L+t\alpha+\alpha-\mu)(L+t\alpha+2\alpha-\mu)\dots (L+t\alpha+(\lfloor \mu \rfloor_\alpha -1)\alpha-\mu)}{\Gamma_{\alpha}(L+t\alpha+\lfloor \mu \rfloor_\alpha\alpha -\mu)}
\\
&=&
\frac{\Gamma_{\alpha}(L+t\alpha)(L+t\alpha+\alpha-\mu)(L+t\alpha+2\alpha-\mu)\dots (L-\{\mu\}_\alpha+(t -1)\alpha )}{\Gamma_{\alpha}(L-\{\mu\}_\alpha+t\alpha)}.
\end{eqnarray*}

Upper and lower bounding the equality above we get

\begin{eqnarray}
\frac{\Gamma_{\alpha}(L+t\alpha)}{\Gamma_{\alpha}(L-\mu+(t+1)\alpha)}
&\geq&
\frac{\Gamma_{\alpha}(L+t\alpha)}{\Gamma_{\alpha}(L-\{\mu\}_\alpha+t\alpha)}
(L-\mu+(t+1)\alpha)^{\lfloor \mu \rfloor_\alpha -1},
\label{eq:lb_1_lower}
\\
\frac{\Gamma_{\alpha}(L+t\alpha)}{\Gamma_{\alpha}(L-\mu+(t+1)\alpha)}
&\leq&
\frac{\Gamma_{\alpha}(L+t\alpha)}{\Gamma_{\alpha}(L-\{\mu\}_\alpha+t\alpha)}
(L-\{\mu\}_\alpha+(t -1)\alpha )^{\lfloor \mu \rfloor_{\alpha} -1}.
\label{eq:lb_1_upper}
\end{eqnarray}

Using \eqref{eq:gautschi_alpha} we have

\begin{equation}
(L+(t-1)\alpha)^{\frac{\{\mu\}_\alpha}{\alpha}}
\leq
\frac{\Gamma_{\alpha}(L+t\alpha)}{\Gamma_{\alpha}(L-\{\mu\}_\alpha+t\alpha)}
\leq
(L+t\alpha)^{\frac{\{\mu\}_\alpha}{\alpha}}.
\label{eq:pom_gautschi}
\end{equation}

Now we are ready to get upper and lower bound on $c_t$:
\begin{eqnarray}
\nonumber
c_t
&\stackrel{\eqref{eq:c_eq_using_gamma_alpha}}{=}&
\frac{\Gamma_{\alpha}(L-\mu+\alpha)}{\Gamma_{\alpha}(L)}
\frac{\Gamma_{\alpha}(L+t\alpha)}{\Gamma_{\alpha}(L-\mu+(t+1)\alpha)}
\\
\nonumber
&\stackrel{\eqref{eq:lb_1_lower}}{\geq}&
\frac{\Gamma_{\alpha}(L-\mu+\alpha)}{\Gamma_{\alpha}(L)}
\frac{\Gamma_{\alpha}(L+t\alpha)}{\Gamma_{\alpha}(L-\{\mu\}_\alpha+t\alpha)}
(L-\mu+(t+1)\alpha)^{\lfloor \mu \rfloor_\alpha -1}
\\
&\stackrel{\eqref{eq:pom_gautschi}}{\geq}&
\frac{\Gamma_{\alpha}(L-\mu+\alpha)}{\Gamma_{\alpha}(L)}
(L+(t-1)\alpha)^{\frac{\{\mu\}_\alpha}{\alpha}}
(L-\mu+(t+1)\alpha)^{\lfloor \mu \rfloor_\alpha -1}.
\label{eq:big_mu_c_lb}
\\
\nonumber
c_t
&\stackrel{\eqref{eq:c_eq_using_gamma_alpha}}{=}&
\frac{\Gamma_{\alpha}(L-\mu+\alpha)}{\Gamma_{\alpha}(L)}
\frac{\Gamma_{\alpha}(L+t\alpha)}{\Gamma_{\alpha}(L-\mu+(t+1)\alpha)}
\\
\nonumber
&\stackrel{\eqref{eq:lb_1_upper}}{\leq}&
\frac{\Gamma_{\alpha}(L-\mu+\alpha)}{\Gamma_{\alpha}(L)}
\frac{\Gamma_{\alpha}(L+t\alpha)}{\Gamma_{\alpha}(L-\{\mu\}_\alpha+t\alpha)}
(L-\{\mu\}_\alpha+(t -1)\alpha )^{\lfloor \mu \rfloor_{\alpha} -1}
\\
&\stackrel{\eqref{eq:pom_gautschi}}{\leq}&
\frac{\Gamma_{\alpha}(L-\mu+\alpha)}{\Gamma_{\alpha}(L)}
(L+t\alpha)^{\frac{\{\mu\}_{\alpha}}{\alpha}}
(L-\{\mu\}_\alpha+(t -1)\alpha )^{\lfloor \mu \rfloor_{\alpha} -1}
\label{eq:big_mu_c_ub}
\end{eqnarray}

Recall that we have $m_\mu=\max(\alpha,\mu-\alpha)$. Then, we can get the following bound on $C_k:$

\begin{eqnarray}
C_k-c_0
&=&\sum_{t=1}^{k-1}c_t
\nonumber
\\
\nonumber
&\stackrel{\eqref{eq:big_mu_c_lb}}{\geq} &
\sum_{t=1}^{k-1}
\frac{\Gamma_{\alpha}(L-\mu+\alpha)}{\Gamma_{\alpha}(L)}
(L+(t-1)\alpha)^{\frac{\{\mu\}_\alpha}{\alpha}}
(L-\mu+(t+1)\alpha)^{\lfloor \mu \rfloor_\alpha -1}
\\
\nonumber
&\stackrel{\eqref{eq:m_mu_def}}{\geq}&
\sum_{t=1}^{k-1}
\frac{\Gamma_{\alpha}(L-\mu+\alpha)}{\Gamma_{\alpha}(L)}
(L-m_{\mu}+t\alpha)^{\frac{\{\mu\}_\alpha}{\alpha}}
(L-m_{\mu}+t\alpha)^{\lfloor \mu \rfloor_\alpha -1}
\\
\nonumber
&=&
\sum_{t=1}^{k-1}
\frac{\Gamma_{\alpha}(L-\mu+\alpha)}{\Gamma_{\alpha}(L)}
(L-m_{\mu}+t\alpha)^{\frac{\mu}{\alpha}-1}
\\
\nonumber
&=&
\frac{\Gamma_{\alpha}(L-\mu+\alpha)}{\Gamma_{\alpha}(L)}
\sum_{t=1}^{k-1}
(L-m_{\mu}+t\alpha)^{\frac{\mu}{\alpha}-1}
\\
\nonumber
&\stackrel{(*)}{\geq} & 
\frac{\Gamma_{\alpha}(L-\mu+\alpha)}{\Gamma_{\alpha}(L)}
\int_{0}^{k-1}
(L-m_{\mu}+t\alpha)^{\frac{\mu}{\alpha}-1}
 dt
\\
\nonumber
&= &
\frac{\Gamma_{\alpha}(L-\mu+\alpha)}{\Gamma_{\alpha}(L)}
\int_{0}^{(k-1)\alpha} 
(L-m_{\mu}+t)^{\frac{\mu}{\alpha}-1}\frac{1}{\alpha}
 dt
\\
\nonumber
&= &
\frac{\Gamma_{\alpha}(L-\mu+\alpha)}{\Gamma_{\alpha}(L)}\frac{1}{\alpha}
\Big[\frac{(L-m_\mu+t)^{\frac{\mu}{\alpha}}}{\frac{\mu}{\alpha}} \Big]_{t=0}^{(k-1)\alpha}
\\
&=&
\frac{\Gamma_{\alpha}(L-\mu+\alpha)}{\Gamma_{\alpha}(L)}
\frac{(L-m_{\mu}+(k-1)\alpha)^{\frac{\mu}{\alpha}}-(L-m_{\mu})^{\frac{\mu}{\alpha}}}{\mu}.
\label{eq:big_mu_C}
\end{eqnarray}

Inequality $(*)$ holds since $(L-m_\mu+t\alpha)^{\mu/\alpha-1}$ is increasing function. 
Note that in the case when $\alpha=\mu$, all bounds above hold with equality and we have
\[
C_k= k.
\]

To finish the proof of the second and third part of the Lemma, it remains to upper bound $\sum_{t=0}^{k-1} c_tL_t^{-1}$. Firstly, note that

\begin{eqnarray}
\nonumber
\frac{c_t}{L_t}
&\stackrel{\eqref{eq:big_mu_c_ub}}{\leq} &
\frac{\Gamma_{\alpha}(L-\mu+\alpha)}{\Gamma_{\alpha}(L)}
(L+t\alpha)^{\frac{\{\mu\}_{\alpha}}{\alpha}}
(L-\{\mu\}_\alpha+(t -1)\alpha )^{\lfloor \mu \rfloor_{\alpha} -1}
(L+t\alpha)^{-1}
\\
\nonumber
&\stackrel{(*)}{\leq} &
\frac{\Gamma_{\alpha}(L-\mu+\alpha)}{\Gamma_{\alpha}(L)}(L+t\alpha)^{\frac{\mu}{\alpha}-1}(L+t\alpha)^{-1}
\\
&= &
\frac{\Gamma_{\alpha}(L-\mu+\alpha)}{\Gamma_{\alpha}(L)}(L+t\alpha)^{\frac{\mu}{\alpha}-2}.
\label{eq:c_over_l_ub}
\end{eqnarray}

Inequality $(*)$ holds due to the fact that $L-\{\mu\}_\alpha+(t -1)\alpha\leq L+t\alpha$.  We can continue bounding as follows

\begin{eqnarray}
\nonumber
\sum_{t=1}^{k-1}\frac{c_t}{L_t}
&\stackrel{\eqref{eq:c_over_l_ub}}{\leq} &
\sum_{t=1}^{k-1}\frac{\Gamma_{\alpha}(L-\mu+\alpha)}{\Gamma_{\alpha}(L)}(L+t\alpha)^{\frac{\mu}{\alpha}-2}
\\
&=&
\nonumber
\frac{\Gamma_{\alpha}(L-\mu+\alpha)}{\Gamma_{\alpha}(L)}
\sum_{t=1}^{k-1}
(L+t\alpha)^{\frac{\mu}{\alpha}-2}
\\
&\stackrel{(*)}{\leq} & 
\nonumber
\frac{\Gamma_{\alpha}(L-\mu+\alpha)}{\Gamma_{\alpha}(L)}
\int_{0}^{k}
(L+t\alpha)^{\frac{\mu}{\alpha}-2}
 dt
\\
&=& 
\nonumber
\frac{\Gamma_{\alpha}(L-\mu+\alpha)}{\Gamma_{\alpha}(L)}
\int_{0}^{k\alpha}
(L+t)^{\frac{\mu}{\alpha}-2}
 \frac{1}{\alpha}dt
\\
&\stackrel{(**)}{=}& 
\label{eq:big_mu_sum_cl}
\begin{cases}
 \frac{\log(L+k\mu)-\log(L)}{\mu}
&\text{if $\alpha=\mu$,}
\\ 
\frac{\Gamma_{\alpha}(L-\mu+\alpha)}{\Gamma_{\alpha}(L)}
   \frac{(L+k\alpha)^{\frac{\mu}{\alpha}-1}-L^{\frac{\mu}{\alpha}-1}}{\mu-\alpha} 
  &\text{if $\alpha<\mu $.}
  \end{cases}
\end{eqnarray}

Inequality $(*)$ holds due to the fact that for $\mu\geq 2\alpha$ we have \[\sum_{t=1}^{k-1}
(L+t\alpha)^{\frac{\mu}{\alpha}-2} \leq \int_{1}^{k}
(L+t\alpha)^{\frac{\mu}{\alpha}-2}
 dt\]
and for $\mu< 2\alpha$ we have 
\[\sum_{t=1}^{k-1}
(L+t\alpha)^{\frac{\mu}{\alpha}-2} \leq \int_{0}^{k-1}
(L+t\alpha)^{\frac{\mu}{\alpha}-2}
 dt.
 \]
Equality $(**)$ holds since

\[
\int_{0}^{k\mu}
(L+t)^{-1}
 \frac{1}{\mu}dt
 =
 \frac{1}{\mu}
 \left[
 \log(L+t) 
 \right]_{t=0}^{k\mu}
 =
 \frac{\log(L+k\mu)-\log(L)}{\mu}
\]
and 
\[
\int_{0}^{k\alpha}
(L+t)^{\frac{\mu}{\alpha}-2}
 \frac{1}{\alpha}
 dt
 =
 \frac{1}{\alpha}
 \left[
 \frac{(L+t)^{\frac{\mu}{\alpha}-1}}{\frac{\mu}{\alpha}-1}
 \right]_{t=0}^{k\alpha}
 =
  \frac{(L+k\alpha)^{\frac{\mu}{\alpha}-1}-L^{\frac{\mu}{\alpha}-1}}{\mu-\alpha}
\]
for $\alpha<\mu$.

To finish the proof, let us now consider the special case when $\alpha=\tfrac{\mu}{2}$ (in other words $L_t=L+t\tfrac{\mu
}{2}$). Note that we have 
\[
\frac{\Gamma_{\alpha}(L-\mu+\alpha)}{\Gamma_{\alpha}(L)}
=
\frac{\Gamma_{\alpha}(L-\alpha)}{\Gamma_{\alpha}(L)}
=
\frac{1}{L-\alpha}
=
\frac{1}{L-\frac{\mu}{2}}
.
\]
Thus, according to \eqref{eq:big_mu_C} and \eqref{eq:big_mu_sum_cl} we have
\begin{eqnarray}
\nonumber
C_k 
&
\stackrel{\eqref{eq:big_mu_C}}{\geq}  
&
1+
\frac{1}{L-\frac{\mu}{2}}
\frac{(L-m_{\mu}+(k-1)\frac{\mu}{2})^2-(L-m_{\mu})^2}{\mu}
\quad
=
\quad
1+
\frac{(L+(k-2)\frac{\mu}{2})^2-(L-\frac{\mu}{2})^2}{(L-\frac{\mu}{2})\mu}
\\
&=&
\frac{(L+(k-2)\frac{\mu}{2})^2-(L-\frac{\mu}{2})^2+(L-\frac{\mu}{2})\mu}{(L-\frac{\mu}{2})\mu}
\label{eq:mu_2_C}
\end{eqnarray}

\begin{eqnarray}
\sum_{t=0}^{k-1}
\frac{c_t}{L_t}
\quad
\stackrel{\eqref{eq:big_mu_sum_cl}}{\leq}
\quad
\frac1L+
\frac{1}{L-\frac{\mu}{2}}
   \frac{(L+k\frac{\mu}{2})^{1}-L^{1}}{\frac{\mu}{2}}
   \quad
   =
   \quad
   \frac1L+
   \frac{k}{L-\frac{\mu}{2}}.
   \label{eq:mu_2_cl}
\end{eqnarray}

Combining \eqref{eq:mu_2_C}, \eqref{eq:mu_2_cl} with Theorem \ref{t: SGD_Ctheorem} we obtain 

 \begin{eqnarray*}
 \sum_{t=1}^k \frac{c_{t-1}}{C_k} \Exp{f(x_t)-f(x_*)}
 &\leq &
\frac{(L-\mu)D_h(x_*,x_0)}{\frac{(L+(k-2)\frac{\mu}{2})^2-(L-\frac{\mu}{2})^2+(L-\frac{\mu}{2})\mu}{(L-\frac{\mu}{2})\mu}}
+
\sigma^2 \frac{\frac1L+ \frac{k}{L-\frac{\mu}{2}}}{\frac{(L+(k-2)\frac{\mu}{2})^2-(L-\frac{\mu}{2})^2+(L-\frac{\mu}{2})\mu}{(L-\frac{\mu}{2})\mu}}
\\
&=&
\frac{(L-\mu)(L-\frac{\mu}{2})\mu D_h(x_*,x_0)+\sigma^2\mu (1-\frac{\mu}{2L }+k) }{(L+(k-2)\frac{\mu}{2})^2-(L-\frac{\mu}{2})^2+(L-\frac{\mu}{2})\mu}
\end{eqnarray*}

which concludes the proof.
\end{proof}

\newpage
\section{Notation Glossary}

\begin{table}[!h]
\begin{center}
\begin{tabular}{|c|l|c|}
\hline
\multicolumn{3}{|c|}{{\bf Standard} }\\
\hline
$\R$ & set of real numbers &\\
$\R^n_+$ & set of positive vectors in $\R^n$ &\\ 
$\Expnot$ & Expectation &\\
$\Prob$ & Probability &\\
$\log$ & natural logarithm & \\
$\langle \cdot ,\cdot \rangle$ & Euclidean inner product  &\\
$\| \cdot \|$ & standard Euclidean norm &\\
$D_h(x,y)$ & Bregman distance between $x,\ y$ & \eqref{eq:bregman_def}\\
$\Gamma_a$ & generalization of the Gamma function & \eqref{eq:Gamma_def} \\

 \hline
 \multicolumn{3}{|c|}{{\bf Global }}\\
 \hline
 $f$ & objective to be minimized over set $Q\subseteq \R^{n}$ & \eqref{Eq:optimization_problem}\\
$x_*$ & minimizer of $f$ over $Q$ & \eqref{Eq:optimization_problem} \\
$\nabla f(x)$ & gradient of $f$ at $x$ &\\
$h$ & reference function, $f$ is rel-smooth with respect to $h$ & \eqref{eq:relsmooth} \\
$L$ & smoothness parameter ($f$ is $L$-smooth relative to $h$) &\eqref{eq:relsmooth} \\
$\mu$ & strong convexity parameter ($f$ is $\mu$-strongly convex relative to $h$) & \eqref{eq:relsc}\\
$x_{(t+1,*)}$ &  next iterate from Algorithm \algD & \\
$x^{(i)}$ & $i$-th coordinate of $x\in \R^n$ & \\
$\ones$ & $n$ dimensional vector of ones & \\
$\ones^i$ & $i$--th column of $n\times n$ identity matrix & \\

 \hline
 \multicolumn{3}{|c|}{{\bf relRCD} (Section~\ref{S:rcd})}\\
 \hline
$\mb$ & minibatch size & \\
$\alpha(h)$ & symmetry measure & \eqref{eq:alfa} \\
$\hat{S}$ & a random subset of $\{1,2,\dots,n\}$ &  \\
$p_0$ & scalar such that $\Prob(i\in \hat{S})=p_0$ for all $i = 1,2,\dots,n$ &  \\
$v$ & parameter vector for ESO & \eqref{eq:eso_def} \\
$w$ & parameter vector for strong convexity & \eqref{eq:sc_vec_def} \\
$\Delta$ & $\min_{i} w^{(i)}/ v^{(i)}$ & \eqref{eq:sc_vec_def} \\

 \hline
 \multicolumn{3}{|c|}{{\bf relSGD } (Section~\ref{S:sgd})}\\
 \hline 
$L_t$ & stepsize controlling parameter & \\
$c_t$ & technical tool for analysis & \\
$\sigma^2$ & global bound on $L_t \Exp{\langle \nabla f(x_t)- \widetilde{g}_t, x_{t+1} -x_{t+1}\rangle \;|\; x_t} $ & \eqref{eq:sgd_variance_assump}\\
$\alpha$ & increase rate of $L_t$ in Lemma~\ref{l:sgd_dec_no_mu}&  \\

\hline
\end{tabular}
\end{center}
\caption{Summary of frequently used notation.}
\label{tbl:notation}
\end{table}

\end{document}

%% file: everything.tex
\usepackage{fullpage}
\usepackage{layout}
\usepackage{multirow}

\usepackage{amsfonts}
\usepackage{amsmath}
\usepackage{amsthm}
\usepackage{amssymb} 

\usepackage{mdframed} 
\usepackage{thmtools}
\usepackage{enumitem}

\definecolor{shadecolor}{gray}{0.95}
\declaretheoremstyle[
headfont=\normalfont\bfseries,
notefont=\mdseries, notebraces={(}{)},
bodyfont=\normalfont,
postheadspace=0.5em,
spaceabove=1pt,
mdframed={
  skipabove=8pt,
  skipbelow=8pt,
  hidealllines=true,
  backgroundcolor={shadecolor},
  innerleftmargin=4pt,
  innerrightmargin=4pt}
]{shaded}

 \usepackage{xcolor}
\usepackage{color}
\usepackage{graphicx}

\usepackage{algorithm}
\usepackage[noend]{algpseudocode}
\usepackage{verbatim}

\usepackage[algo2e]{algorithm2e} 

\newcommand{\R}{\mathbb{R}} 

\newcommand{\cO}{{\cal O}}

\usepackage[colorinlistoftodos,bordercolor=orange,backgroundcolor=orange!20,linecolor=orange,textsize=scriptsize]{todonotes}

\newcommand{\mb}{\tau}

\newcommand{\eqdef}{\overset{\text{def}}{=}} 

\DeclareMathOperator{\argmin}{argmin}        

\newcommand{\Expnot}{{\bf E}}
\newcommand{\Prob}{{\bf P}}

\newcommand{\Exp}[1]{{\bf E}\left[#1\right] }    

\declaretheorem[style=shaded,within=section]{definition}
\declaretheorem[style=shaded,sibling=definition]{theorem}
\declaretheorem[style=shaded,sibling=definition]{proposition}
\declaretheorem[style=shaded,sibling=definition]{assumption}
\declaretheorem[style=shaded,sibling=definition]{corollary}
\declaretheorem[style=shaded,sibling=definition]{lemma}

\theoremstyle{remark}
\newtheorem{remark}{Remark} 
\usepackage[colorlinks=true,linkcolor=blue]{hyperref}